\documentclass[review]{elsarticle}

\usepackage{lineno,hyperref}
\modulolinenumbers[5]
\journal{J. Comput. Phys.}

\usepackage{amsmath,amssymb,amsthm,amsfonts,afterpage,float,bm,stmaryrd,paralist,wrapfig}

\usepackage{cancel}
\usepackage{color}
\usepackage{fancyhdr}
\usepackage{graphics}
\usepackage{hyperref}
\usepackage{graphicx,epsf}
\usepackage{makeidx}
\usepackage{epsfig}
\usepackage{lscape}

\usepackage{empheq,xcolor}

\pssilent

\newtheorem{theorem}{Theorem}[section]
\newtheorem{lemma}{Lemma}[section]

\newtheorem{definition}{Definition}[section]
\newtheorem{remark}{Remark}[section]

\numberwithin{equation}{section}
\numberwithin{figure}{section}
\numberwithin{table}{section}

\def\XXint#1#2#3{{\setbox0=\hbox{$#1{#2#3}{\int}$}
\vcenter{\hbox{$#2#3$}}\kern-.51\wd0}}

\def\non{\nonumber }

\newcommand\eps{{\epsilon}}

\newcommand{\ud}{\mathrm{d}}

\newcommand{\I}{\mathrm{I}}
\newcommand{\TT}{\mathbb{T}^2}

\graphicspath{{Figure/}}

\begin{document}

\setlength{\pdfpageheight}{\paperheight}
\setlength{\pdfpagewidth}{\paperwidth}
\title{Maximum Principle Preserving Schemes for Binary Systems with Long-range Interactions}
\author{Xiang Xu}
\address{Department of Mathematics and Statistics, Old Dominion University, Norfolk, VA 23529}
\author{Yanxiang Zhao}
\address{Department of Mathematics, George Washington University, Washington D.C., 20052}

\begin{abstract}
We study some maximum principle preserving and energy stable schemes for the Allen-Cahn-Ohta-Kawasaki model with fixed volume constraint. With the inclusion of a nonlinear term $f(\phi)$ in the Ohta-Kawasaki free energy functional, we show that the Allen-Cahn-Ohta-Kawasaki dynamics is maximum principle preserving. We further design some first order energy stable numerical schemes which inherit the maximum principle preservation in both semi-discrete and fully-discrete levels. Furthermore, we apply the maximum principle preserving schemes to a general framework for binary systems with long-range interactions. We also present some numerical results to support our theoretical findings.
\end{abstract}


\date{\today}
\maketitle


\section{Introduction}\label{sec:Introduction}

Ohta-Kawasaki (OK) model is introduced in \cite{OhtaKawasaki_Macromolecules1986} and has been extensively applied for the study of phase separation of diblock copolymers, which
have generated much interest in materials science in the past years due to their remarkable ability for self-assembly into nanoscale ordered structures \cite{Hamley_Wiley2004}. Diblock copolymers are chain molecules made by two different segment species, say $A$ and $B$ species. Due to the chemical incompatibility, the two species tend to be phase-separated; on the other hand, the two species are connected by covalent chemical bonds, which leads to the so-called microphase separation. The OK model can describe such microphase separation for diblock copolymers via a free energy functional:
\begin{align}\label{functional:OK}
E^{\text{OK}}[\phi] = \int_{\mathbb{T}^d} \dfrac{\epsilon}{2}|\nabla\phi|^2 + \dfrac{1}{\epsilon}W(\phi)\ \text{d}x + \dfrac{\gamma}{2}\int_{\mathbb{T}^d} |(-\Delta)^{-\frac{1}{2}}(f(\phi)-\omega)|^2\ \text{d}x,
\end{align}
with a volume constraint
\begin{align}\label{eqn:Volume}
\int_{\mathbb{T}^d} (f(\phi) - \omega)\ \text{d}x = 0.
\end{align}
Here $\mathbb{T}^d = \prod_{i=1}^d [-X_i, X_i] \subset \mathbb{R}^d, d = 2, 3$ denotes a periodic box and $0<\epsilon \ll 1$ is an interface parameter that indicates the system is in deep segregation regime. $\phi = \phi(x)$ is a phase field labeling function which represents the concentration of $A$ species. By the assumption of incompressibility for the binary system, the concentration of $B$ species can be implicitly represented by $1-\phi(x)$. Function $W(\phi) = 18(\phi^2-\phi)^2$ is a double well potential which enforces the phase field function $\phi$ to be equal to 1 inside the interface and 0 outside the interface. Near the interfacial region, the phase field function $\phi$ rapidly but smoothly transitions from 0 to 1. A new term of $f(\phi) = 3\phi^2 - 2\phi^3$ is introduced in the free energy functional to resemble $\phi$ as the indicator for the $A$ species. The first integral in (\ref{functional:OK}) is a local surface energy which represents the short-range interaction between the chain molecules and favors the large domain; while the second integral in (\ref{functional:OK}) is a term for the long-range (nonlocal) repulsive interaction with $\gamma >0$ being the strength of the repulsive force. Finally, $\omega\in(0,1)$ is the relative volume of the $A$ species.

To study the microphase separation and the pattern formation of the diblock copolymer, we consider the $L^2$ gradient flow dynamics of the OK model. On the other hand, to relax the volume constraint (\ref{eqn:Volume}), we can incorporate a penalty term into (\ref{functional:OK}) and change it into an unconstrained one:
\begin{align}\label{functional:pOK}
E^{\text{pOK}}[\phi] = \int_{\mathbb{T}^d} \dfrac{\epsilon}{2}|\nabla\phi|^2 + \dfrac{1}{\epsilon}W(\phi)\ \text{d}x + \dfrac{\gamma}{2}\int_{\mathbb{T}^d}  |(-\Delta)^{-\frac{1}{2}}(f(\phi)-\omega)|^2\ \text{d}x + \dfrac{M}{2}\left[ \int_{\mathbb{T}^d}  f(\phi)-\omega\ \text{d}x \right]^2,
\end{align}
where $M\gg1$ is a penalty constant. Then we can consider the corresponding penalized $L^2$ gradient flow dynamics with given initial $\phi(x,t=0) = \phi_0$, which thereafter is called penalized Allen-Cahn-Ohta-Kawasaki (pACOK) equation:
\begin{align}\label{eqn:pACOK}
\dfrac{\partial}{\partial t} \phi = \epsilon\Delta\phi - \dfrac{1}{\epsilon}W'(\phi) - \gamma(-\Delta)^{-1}(f(\phi)-\omega)f'(\phi) - M\int_{\mathbb{T}^d} (f(\phi)-\omega)\ \text{d}x\cdot f'(\phi).
\end{align}

Our main contribution in this paper is threefold. Firstly, the new form of $f(\phi)$ guarantees that the pACOK equation is maximum principle preserving (MPP), namely, if the initial data is bounded $0\le\phi_0\le1$, then $0\le\phi(x)\le1$ for any later time. Secondly, we adopt a linear splitting method to the pACOK equation and then apply a semi-implicit scheme for the numerical simulations. This scheme treats the linear terms implicitly but all the nonlinear and nonlocal terms explicitly, and with proper choice for the splitting constant, it inherits the MPP property at both time-discrete and fully-discrete levels. Besides, just as the energy dissipation law (energy stability) is obeyed by the continuous $L^2$ gradient flow (\ref{eqn:pACOK}), the proposed numerical scheme also successfully inherits the energy stability at both time-discrete and fully-discrete levels. Thirdly, the error estimate analysis is carried and the rate of convergence is verified by numerical simulations.

The inclusion of a new nonlinear term $f(\phi)$ in the OK model makes the key novelty in this paper. On one hand, this term $f(\phi)$ accounts for mimicking the behavior of $\phi$ as the indicator for the $A$-rich region, so it satisfies the condition:
\begin{align}\label{eqn:f_condition}
f(0) = 0, \quad f(1) = 1.
\end{align}
On the other hand, in order to pin the phase field label $\phi$ at 1 and 0 inside and outside the $A$-$B$ interface, respectively, we set an extra condition
\begin{align}\label{eqn:f'_condition}
f'(0) = 0, \quad f'(1) = 0.
\end{align}
so that the evolution of the pACOK dynamics (\ref{eqn:pACOK}) only updates the phase field $\phi$ near the interface but not in the away-from-interface region. This will help maintain $\phi$ as a desired tanh profile better than simply taking $f(\phi) = \phi$.  See \cite{Zhao_2018CMS, WangRenZhao_CMS2019, XuZhao_JSC2019} for the numerical comparisons between linear and nonlinear choices of $f(\phi)$. The polynomial of the smallest degree satisfying both conditions (\ref{eqn:f_condition}) and (\ref{eqn:f'_condition}) is
\begin{align}\label{eqn:f}
f(\phi) = 3\phi^2 - 2\phi^3.
\end{align}
In some scenario, a polynomial of higher degree might be required. For instance, in the study of energy stable numerical scheme based on operator splitting, one needs to perform a linear extension for the nonlinearity $f$ up to second order continuous derivative, then the minimal degree has to be fifth \cite{XuZhao_JSC2019}. The real magic that $f(\phi)$ in (\ref{eqn:f}) plays is that it preserves the maximum principle. The key observation is that $f'$ and $W'$ share a common factor $\phi - \phi^2$, such that any possible growth on $\phi$ (which potentially breaks the MPP) can be safely killed by the double well potential term to save the MPP. See the proof of Theorem \ref{theorem:MPP_continuous} for details.


The MPP is an important property held by the Allen-Cahn equation. It says that if the initial data is bounded between 0 and 1, then the solution remains between 0 and 1 for any later time. In recent years, efforts have been devoted to investigate the MPP numerical schemes for the Allen-Cahn equation. Tang and Yang studied a first order implicit-explicit scheme for the MPP property for the Allen-Cahn equation in \cite{TangYang_JCM2016}. They further extended the results to the generalized Allen-Cahn equation in \cite{ShenTangYang_CMS2016}. Later, some attempts have been made to study second order MPP schemes for fractional-in-space Allen-Cahn equation \cite{HouTangYang_JSC2017} and nonlocal Allen-Cahn equation \cite{DuJuLiQiao_SINA2019}. Recently some adaptive second order MPP schemes have been considered for the Allen-Cahn equation \cite{LiaoTangZhou_SINA2019} and time-fractional Allen-Cahn equation \cite{LiaoTangZhou_JCP2019}.

In this paper, as a first attempt, we will explore the MPP scheme for an Allen-Cahn type dynamics with a long-range interaction term. The new ingredient in the MPP scheme is inspired by the continuous MPP property, namely, the nonlinear function $f(\cdot)$ which satisfies the condition (\ref{eqn:f'_condition}). However two things are different between the continuous  and the discrete settings.  One is that $f(\cdot)$ has to be linearly extended to 0 and 1 in the continuous case but not in the discrete case. See the equation (\ref{eqn:f_extension}) and the related discussion. The other is that in the continuous case, $f(\cdot)$ has to be of the smallest degree to satisfy (\ref{eqn:f'_condition}) in order to be well controlled by the double well potential $W$, while in the discrete case, any polynomial $f(\cdot)$ satisfying (\ref{eqn:f'_condition}) (and (\ref{eqn:f_condition})) would do the trick. Allowing weaker conditions for the discrete schemes is due to the fact  that the MPP of $\phi(t)$ in the continuous pACOK dynamics depends on the entire history before $t$; while the discrete MPP of $\phi^{n}$ only depends on the $k$ previous states $\{\phi^{j}\}_{j=n-k}^{n-1}$ (in this paper, we focus on the the case of $k=1$). See the proofs of Theorem \ref{theorem:MPP_continuous}, Theorem \ref{theorem:MPP_semi} and Theorem \ref{theorem:MPP_fully} for details. This indeed provides much flexibility on choosing $f(\cdot)$ to exploit various discrete MPP schemes for Allen-Cahn type dynamics with long-range interactions.

For the discrete MPP schemes, the Lemma \ref{lemma:MPP_condition} (for the time-discrete case) and Lemma \ref{lemma:MPP_condition2} (for the fully-discrete case) play the key role which will be crucial for the analysis of not only the first order MPP scheme in this paper but also potentially for other higher order ones. Plus, these two lemmas suggest that the nonlocal terms might have to be treated explicitly in order to satisfy the discrete MPP.

Our work is by no mean an additional extension of the existing work on MPP by changing from one model to another. This work has potential wider impact on many other applications. Indeed, we further extend this model to binary systems with various long-range interactions. See Section \ref{section:GeneralFramework} for the detailed discussion on the extension. Our work could provide a general framework to explore the MPP numerical schemes for other applications such as the micromagnetic model for garnet films \cite{CondetteMelcherSuli_MathComp2010} , FitzHugh-Nagumo system\cite{RenTruskinovsky_Elasticity2000} , implicit solvation model\cite{Zhao_2018CMS} etc.

Since discrete energy stability is a byproduct when exploring the MPP schemes, we briefly review some of the existing work for the energy stable numerical methods. The energy stable schemes, first studied by Du and Nicolaides in \cite{DuNicolaides_SINA1991} for a second order accurate unconditionally stable time-stepping scheme for the Cahn-Hilliard equation, has been extensively studied for various $L^2$ and $H^{-1}$ gradient flow dynamics such as the standard Allen-Cahn and Cahn-Hilliard equations \cite{ShenYang_DCDSA2010}, phase field crystal model \cite{WiseWangLowengrub_SINA2009, HuWiseWangLowengrub_JCP2009}, modified phase field crystal model \cite{WangWise_SINA2011}, and epitaxial thin film growth model \cite{ChenCondeWangWangWise_JSC2012} etc. Several popular numerical schemes adopted by the community are listed below. One is the convex splitting method \cite{Eyre_Proc1998} in which the double well potential $W(\phi)$ is split into the sum of a convex function and a concave one, and the convex part is treated implicitly and the concave one is treated explicitly. However, a nonlinear system usually needs to be solved at each time step which induces high computational cost. Another widely adopted method is the stabilized semi-implicit method \cite{XuTang_SINA2006,ShenYang_DCDSA2010} in which $W(\phi)$ is treated explicitly. A linear stabilizing term is added to maintain the energy stability.  Another recent method is the IEQ method \cite{ChengYangShen_JCP2017, Yang_JCP2016} in which all nonlinear terms are treated semi-implicitly, the energy stability is preserved and the resulting numerical schemes lead to a symmetric positive definite linear system to be solved at each time step. A variation of the IEQ method, which is called SAV method, is well studied in the last couple of years \cite{ShenXuYang_SIAMReview2019}. For a more comprehensive review on the topics of the modeling and numerical methods of phase field approach, we refer the interested readers to \cite{DuFeng_Handbook2020}.

Some conventional notations adopted throughout the paper are collected here. We will denote by $\|\cdot\|_{L^p}$ and $\|\cdot\|_{H^s}$ the standard norms for the periodic Sobolev spaces $L^p_{\text{per}}(\mathbb{T}^d)$ and $H^s_{\text{per}}(\mathbb{T}^d)$. The standard $L^2$ inner product will be denoted by $\langle \cdot, \cdot \rangle$. In order to make the MPP satisfied by the pACOK equation (\ref{eqn:pACOK}), the nonlinear function $f$ needs to be extended to $\tilde{f}$ as follows:
\begin{align}\label{eqn:f_extension}
\tilde{f} =
\begin{cases}
0, \hspace{0.64in} s < 0; \\
3s^2 - 2s^3, \quad 0 \le s \le 1; \\
1, \hspace{0.64in} s > 1.
\end{cases}
\end{align}
We still use $f$ to denote such an extension for the brevity of notations. Indeed, the extension of $f$ is only used for the proof of the MPP property for the continuous pACOK equation (\ref{eqn:pACOK}). For the time-discrete and fully-discrete pACOK equations, the unextended $f(s) = 3s^2 - 2s^3$ suffices to guarantee the MPP and energy stability, see the proofs of Theorem \ref{theorem:MPP_continuous}, Theorem \ref{theorem:MPP_semi} and Theorem \ref{theorem:MPP_fully} for the details. We take
\[
L_{W''}: = \|W''\|_{L^{\infty}[0,1]}, \quad L_{f''}: = \|f''\|_{L^{\infty}[0,1]}, \quad L_{f'} = \|f'\|_{L^{\infty}[0,1]}.
\]
Next, $\|(-\Delta)^{-1}\|$ denotes the optimal constant such that $\|(-\Delta)^{-1}f\|_{L^{\infty}}\le C \|f\|_{L^{\infty}}$, namely, it is the norm of the operator $(-\Delta)^{-1}$ from $L^{\infty}(\mathbb{T}^d)$ to itself. We will take $[\![ n]\!]$ to be the set of integers $\{1,2,\cdots, n\}$. Lastly, we denote $\tilde{\omega} = \max\{\omega, 1-\omega\}$.

The rest of the paper is organized as follows. In Section 2, we will prove the MPP property for the continuous pACOK dynamics. In Section 3, a first order time-discrete numerical scheme will be studied which inherits the MPP and energy stability. In Section 4, we will conduct analysis of MPP and energy stability for the fully-discrete scheme. The error estimate will be carried as well. The extension of the MPP to  general binary systems with long-range interactions is discussed in Section 5.  We will present some numerical results to support our theoretical findings in Section 6, followed by a summary in Section 7. In the appendix, we present the wellposedness of the pACOK equation and the $L^{\infty}$ bound for the weak solution of the pACOK equation.

\section{MPP for the continuous pACOK dynamics}

In this section, we will prove that the continuous pACOK equation (\ref{eqn:pACOK}) satisfies the MPP, and one can see the critical role that $f(\phi)$ plays in the theory. Note that in this section, $f(\cdot)$ represents the extended version (\ref{eqn:f_extension}).

\begin{theorem}\label{theorem:MPP_continuous}
The pACOK equation (\ref{eqn:pACOK}) is MPP, namely, if $0\le \phi_0 \le 1$, then $0\le \phi(t) \le 1$ for any $t>0$, provided that $\phi_0\in H^1(\mathbb{T}^d)$ and
\begin{align}\label{eqn_MPPcondition_continuous}
\frac{\epsilon\tilde{\omega}}{6}\Big[ \gamma\|(-\Delta)^{-1}\| + M |\mathbb{T}^d|  \Big] \le 1.
\end{align}
\end{theorem}
\begin{proof}
Multiplying on the two sides of (\ref{eqn:pACOK}) by $2\phi-1$, one has
\begin{align*}
\dfrac{\partial}{\partial t} (\phi^2-\phi) =&\ \epsilon \Delta(\phi^2-\phi) - 2\epsilon|\nabla\phi|^2 - \dfrac{36}{\epsilon}(\phi^2-\phi)(2\phi-1)^2 \\
&- \gamma(-\Delta)^{-1}(f(\phi)-\omega)\cdot 6(\phi-\phi^2)(2\phi-1) - M\int_{\mathbb{T}^d} (f(\phi)-\omega)\ \text{d}x\cdot 6(\phi-\phi^2)(2\phi-1)
\end{align*}
Multiplying on the two sides of the above equation by $(\phi^2-\phi)^+$ and taking integral over $\mathbb{T}^d$, one has
\begin{align*}
&\dfrac{1}{2}\dfrac{\partial}{\partial t} \int_{\mathbb{T}^d} |(\phi^2-\phi)^+|^2\ \text{d}x \\
=&\ -\epsilon \int_{\mathbb{T}^d} |\nabla(\phi^2-\phi)^+|^2\ \text{d}x - 2\epsilon \int_{\mathbb{T}^d} |\nabla\phi|^2 (\phi^2-\phi)^+\ \text{d}x  - \dfrac{36}{\epsilon}\int_{\mathbb{T}^d}|(\phi^2-\phi)^+|^2 (2\phi-1)^2\ \text{d}x \\
& + 6\gamma \int_{\mathbb{T}^d} (-\Delta)^{-1}(f(\phi)-\omega) |(\phi^2-\phi)^+|^2(2\phi-1)\ \text{d}x  + 6M\int_{\mathbb{T}^d} (f(\phi)-\omega)\ \text{d}x \int_{\mathbb{T}^d} |(\phi^2-\phi)^+|^2(2\phi-1)\ \text{d}x \\
=&\ -\epsilon \int_{\mathbb{T}^d} |\nabla(\phi^2-\phi)^+|^2\ \text{d}x - 2\epsilon \int_{\mathbb{T}^d} |\nabla\phi|^2 (\phi^2-\phi)^+\ \text{d}x  - \dfrac{36}{\epsilon}\int_{\mathbb{T}^d}|(\phi^2-\phi)^+|^2 \Big( (2\phi-1)^2 - (A+B)(2\phi-1) \Big)\ \text{d}x
\end{align*}
where
\[
A = \dfrac{\gamma\epsilon}{6}(-\Delta)^{-1}(f(\phi)-\omega), \quad B = \dfrac{M\epsilon}{6}\int_{\mathbb{T}^d}(f(\phi)-\omega)\ \text{d}x.
\]
When $\phi^2-\phi\ge0$, one has $|2\phi-1| \ge 1$. Note that the condition (\ref{eqn_MPPcondition_continuous}) implies $\|A\|_{L^{\infty}} + |B|\le 1$, therefore $(2\phi-1)^2 - (A+B)(2\phi-1)\ge0$, which implies that
\[
\dfrac{1}{2}\dfrac{\partial}{\partial t} \int_{\mathbb{T}^d} |(\phi^2-\phi)^+|^2\ \text{d}x \le 0.
\]
Taking integral for time from 0 to $t$ leads
\[
\int_{\mathbb{T}^d} |(\phi^2-\phi)^+|^2(t)\ \text{d}x \le \int_{\mathbb{T}^d} |(\phi^2-\phi)^+|^2(0)\ \text{d}x.
\]
If $0\le \phi(0) \le 1$, $\int_{\mathbb{T}^d} |(\phi^2-\phi)^+|^2(0)\ \text{d}x=0$, then
\[
\int_{\mathbb{T}^d} |(\phi^2-\phi)^+|^2(t)\ \text{d}x \le 0 \Rightarrow 0\le \phi(t) \le 1,
\]
which completes the proof.
\end{proof}

\begin{remark}
The wellposedness of the pACOK equation (\ref{eqn:pACOK}) can be well established by using the standard minimization movement scheme, see Theorem \ref{theorem:wellposedness} in the Appendix for the related discussion.
\end{remark}

\begin{remark}
For the condition (\ref{eqn_MPPcondition_continuous}), it is theoretically easy to achieve due to the smallness of the interfacial width $\epsilon$, though the long-range repulsion strength $\gamma$ and the penalty constant $M$ are supposed to be large.
\end{remark}

\begin{remark}
The extension of $f$ is critical in order to bound the $A$ term as
\[
\|A\|_{L^{\infty}} \le \frac{\gamma\epsilon}{6} \|(-\Delta)^{-1}\|\cdot \|f(\phi)-\omega\|_{L^{\infty}} \le \frac{\gamma\epsilon}{6} \|(-\Delta)^{-1}\|\cdot \tilde{\omega}.
\]
On the other hand, in the 2d case, one can still have the MPP held for non-extended $f(\phi)$ by showing that $\|f(\phi)\|_{L^{\infty}} \le C$ for some generic constant $C$ which depends on $\|\phi_0\|_{H^1(\mathbb{T}^2)}, \epsilon^{-1}, \omega, \gamma$ and $M$. See the Theorem \ref{theorem: Linfity} in the Appendix for the $L^{\infty}$ bound for the 2d weak solution $\phi$. Then $A$ is still bounded as
\[
\|A\|_{L^{\infty}} \le \frac{C\gamma\epsilon}{6} \|(-\Delta)^{-1}\|\tilde{\omega}.
\]
However, to bound the quantity $\|A\|_{L^{\infty}}+|B|$ by 1, one has to take sufficiently small value of $\gamma$, which is theoretically acceptable but unrealistic in applications.
\end{remark}

\section{Time-discrete Scheme: MPP and Energy Stability}

Now we will consider a semi-discrete scheme for the pACOK equation (\ref{eqn:pACOK}), and show that such a scheme satisfies the MPP and energy stability under some conditions. Given time interval $[0,T]$ and an integer $N>0$, we take the uniform time step size $\tau = T/N$ and $t_n = n\tau$ for $n=0,1,\cdots,N$. Let $\phi^n(x) \approx \phi(t_n,x)$ be the temporal semi-discrete approximation of the solution $\phi$ at $t_n$. Given initial data $\phi^0 = \phi_0$ and a splitting constant (or stabilizer) $\kappa > 0$, we consider the following stabilized time-discrete scheme:
\begin{align}\label{eqn:pACOK_SemiImplicit}
\left(\dfrac{1}{\tau}+\dfrac{\kappa}{\epsilon}\right)(\phi^{n+1}-\phi^n)  = &\ \epsilon\Delta\phi^{n+1} - \dfrac{1}{\epsilon}W'(\phi^n) \nonumber\\
& - \gamma(-\Delta)^{-1}(f(\phi^n)-\omega)f'(\phi^n) - M\int_{\mathbb{T}^d} (f(\phi^n)-\omega)\text{d}x\cdot f'(\phi^n),
\end{align}
which can be rewritten as
\begin{align}\label{eqn:pACOK_SemiImplicit2}
\left(\left(1+\dfrac{\tau\kappa}{\epsilon}\right)I - \tau\epsilon\Delta \right)\phi^{n+1} = &\left(1+\dfrac{\tau\kappa}{\epsilon}\right)\phi^n - \dfrac{\tau}{\epsilon}W'(\phi^n) \nonumber\\
&- \tau\gamma(-\Delta)^{-1}(f(\phi^n)-\omega)f'(\phi^n) - \tau M\int_{\mathbb{T}^d} (f(\phi^n)-\omega)\text{d}x\cdot f'(\phi^n).
\end{align}
A simple calculation reveals that the eigenvalues of the operator $(1+\tau\kappa\epsilon^{-1})I - \tau\epsilon\Delta$ on the left hand side of (\ref{eqn:pACOK_SemiImplicit2}) are all positive. Therefore the scheme is unconditionally uniquely solvable.

\subsection{MPP for time-discrete scheme}

In this section, we will show that the scheme (\ref{eqn:pACOK_SemiImplicit2}) is MPP. To this end, we begin with a lemma.

\begin{lemma}\label{lemma:MPP_condition}
Let
\[
\mathcal{F}(\psi) = \left(1+\dfrac{\tau\kappa}{\epsilon}\right)\psi - \dfrac{\tau}{\epsilon}W'(\psi) - \tau\gamma(-\Delta)^{-1}(f(\psi)-\omega)f'(\psi) - \tau M\int_{\mathbb{T}^d} (f(\psi)-\omega)\emph{d}x f'(\psi).
\]
If $\psi(x)\in [0,1]$, then we have
\[
\max_{\psi\in[0,1]} \mathcal{F}(\psi) = 1+\dfrac{\tau\kappa}{\epsilon} ; \quad \min_{\psi\in[0,1]} \mathcal{F}(\psi) = 0,
\]
provided that
\begin{align}\label{eqn:MPP_condition}
\frac{1}{\tau}+ \dfrac{\kappa}{\epsilon} \ge \dfrac{L_{W''}}{\epsilon}  + \tilde{\omega}L_{f''} \Big(\gamma\|(-\Delta)^{-1}\|+M|\mathbb{T}^d| \Big).
\end{align}
\end{lemma}
\begin{proof}
Note that $f(\cdot)$ satisfies $f'(0) = f'(1) = 0$, it follows that for $\psi\equiv 0$, $\mathcal{F}(\psi) = 0$; for $\psi\equiv 1$, $\mathcal{F}(\psi) = 1+\tau\kappa/\epsilon$. For any other $\psi$ such that $0\le \psi \le 1$ and any $x\in\mathbb{T}^d$, one has
\begin{align*}
\mathcal{F}(\psi(x)) & =  \mathcal{F}(0) + \left(1+\dfrac{\tau\kappa}{\epsilon}\right)\psi(x) - \frac{\tau}{\epsilon}\psi(x)W''(\xi_0)  \\
& \quad - \tau\gamma \Big( (-\Delta)^{-1}(f(\psi)-\omega) \Big)(x) \cdot \psi(x) f''(\eta_0) - \tau M \int_{\mathbb{T}^d} (f(\psi)-\omega ) \text{d}x \cdot \psi(x) f''(\eta_0) \\
& \ge \mathcal{F}(0) +  \left(1+\dfrac{\tau\kappa}{\epsilon}\right)\psi(x) - \dfrac{\tau}{\epsilon}\psi L_{W''} - \tau\gamma\psi\|(-\Delta)^{-1}\|\tilde{\omega}L_{f''} - \tau M \psi |\mathbb{T}^d| \tilde{\omega} L_{f''} \ge  \mathcal{F}(0),
\end{align*}
where $\xi_0, \eta_0 \in (0,\psi(x)) \subset (0,1)$ are constants obtained from Taylor expansion. On the other hand,
\begin{align*}
\mathcal{F}(1-\psi(x)) & =  \mathcal{F}(1) - \left(1+\dfrac{\tau\kappa}{\epsilon}\right)\psi(x) + \frac{\tau}{\epsilon}\psi(x)W''(\xi_1) \\
&\quad + \tau\gamma \Big( (-\Delta)^{-1}(f(1-\psi)-\omega)\Big)(x)\cdot\psi(x) f''(\eta_1) + \tau M \int_{\mathbb{T}^d}(f(1-\psi)-\omega)\text{d}x\cdot \psi(x) f''(\eta_1) \\
& \le \mathcal{F}(1)  -\left(1+\dfrac{\tau\kappa}{\epsilon}\right)\psi(x) + \dfrac{\tau}{\epsilon}\psi(x) L_{W''} + \tau\gamma\psi(x)\|(-\Delta)^{-1}\| \tilde{\omega}L_{f''} - \tau M \psi |\mathbb{T}^d| \tilde{\omega}L_{f''} \le \mathcal{F}(1),
\end{align*}
where $\xi_1, \eta_1 \in (1-\psi(x), 1) \subset (0,1)$ are constants by Taylor expansion. Consequently we have the desired bounds for $\mathcal{F}(\psi)$.
\end{proof}

Now we present the MPP property for the scheme (\ref{eqn:pACOK_SemiImplicit}) or (\ref{eqn:pACOK_SemiImplicit2}).

\begin{theorem}\label{theorem:MPP_semi}
The stabilized time-discrete semi-implicit scheme (\ref{eqn:pACOK_SemiImplicit}) or (\ref{eqn:pACOK_SemiImplicit2}) is MPP, namely
\[
0\le \phi^0 \le 1 \Rightarrow 0\le \phi^n \le 1, \quad \forall n \in [\![ N]\!].
\]
provided that the condition (\ref{eqn:MPP_condition}) holds.
\end{theorem}
\begin{proof}
We can prove the result by induction. Assume that $0\le \phi^n \le 1$, and $\phi^{n+1}$ is obtained by the scheme (\ref{eqn:pACOK_SemiImplicit2}). Assume $\phi^{n+1}$ reaches the maximal value at $x^*$, then $-\tau\epsilon\Delta\phi^{n+1}(x^*)\ge0$, and
\[
\left(1+\dfrac{\tau\kappa}{\epsilon}\right) \phi^{n+1}(x^*) \le \mathcal{F}(\phi^n(x^*)) \le 1+\dfrac{\tau\kappa}{\epsilon} \Rightarrow \phi^{n+1} \le 1.
\]
Similarly let $x_*$ be a minimal point for $\phi^{n+1}$, then  $-\tau\epsilon\Delta\phi^{n+1}\le0$, and
\[
\left(1+\dfrac{\tau\kappa}{\epsilon}\right) \phi^{n+1}(x_*) \ge \mathcal{F}(\phi^n(x_*)) \ge 0\Rightarrow \phi^{n+1} \ge 0.
\]
which completes the proof.
\end{proof}

\begin{remark}
Note that the condition (\ref{eqn:MPP_condition}) holds for sufficiently large stabilizer $\kappa$ no matter what value of $\tau>0$. Therefore the stabilized semi-implicit scheme (\ref{eqn:pACOK_SemiImplicit}) or (\ref{eqn:pACOK_SemiImplicit2}) is unconditionally MPP for sufficiently large $\kappa$.
\end{remark}

\subsection{Energy stability for time-discrete scheme}

While the stabilized semi-discrete scheme (\ref{eqn:pACOK_SemiImplicit2}) is MPP, it is also energy stable as shown in the following theorem.
\begin{theorem}\label{theorem:EnergyStability}
Assume the initial $\phi^0$ satisfies $0\le \phi^0\le 1$, then the stabilized semi-implicit scheme (\ref{eqn:pACOK_SemiImplicit}) or (\ref{eqn:pACOK_SemiImplicit2}) is unconditionally energy stable in the sense that
\begin{align}
E^{\emph{pOK}}[\phi^{n+1}] \le E^{\emph{pOK}}[\phi^{n}]
\end{align}
 provided that
\begin{align}\label{eqn:ES_condition}
\frac{\kappa}{\epsilon} \ge \frac{L_{W''}}{\epsilon} + (L_{f'}^2+\tilde{\omega}L_{f''})\Big(\gamma\|(-\Delta)^{-1}\| + M |\mathbb{T}^d|\Big).
\end{align}
\end{theorem}
\begin{proof}
Taking the $L^2$ inner product with $\phi^{n+1}-\phi^{n}$ on the two sides of (\ref{eqn:pACOK_SemiImplicit}) , we have
\begin{align}\label{eqn:estimate1}
&\dfrac{1}{\tau} \|\phi^{n+1} - \phi^{n}\|_{L^2}^2 \nonumber\\
= &\  - \dfrac{\kappa}{\epsilon}\|\phi^{n+1} - \phi^{n}\|_{L^2}^2   \underbrace{-\epsilon \langle \nabla\phi^{n+1}, \nabla\phi^{n+1} - \nabla\phi^{n} \rangle }_{\text{I}} \underbrace{ - \epsilon^{-1} \langle W'(\phi^{n}), \phi^{n+1} - \phi^{n}\rangle }_{\text{II}} \nonumber\\
& \underbrace{ -\gamma \left\langle (-\Delta)^{-1}(f(\phi^n)-\omega)f'(\phi^n), \phi^{n+1}-\phi^{n} \right\rangle }_{\text{III}}  \underbrace{ -M \textstyle{\int_{\mathbb{T}^d}} (f(\phi^n)-\omega)\text{d}x \left\langle f'(\phi^n), \phi^{n+1}-\phi^{n} \right\rangle }_{\text{IV}}.
\end{align}
Using the identity $a\cdot (a-b) = \frac{1}{2}|a|^2 - \frac{1}{2}|b|^2 + \frac{1}{2}|a-b|^2$ and $b\cdot (a-b) = \frac{1}{2}|a|^2 - \frac{1}{2}|b|^2 - \frac{1}{2}|a-b|^2$, we have:
\begin{align*}
\text{I} =&\ - \frac{\epsilon}{2} \left( \|\nabla\phi^{n+1}\|_{L^2}^2 - \|\nabla\phi^{n}\|_{L^2}^2 + \|\nabla\phi^{n+1}-\nabla\phi^n\|_{L^2}^2 \right); \\
\text{II} =& - \epsilon^{-1}\left\langle 1, W'(\phi^n)(\phi^{n+1}-\phi^n) \right\rangle = -\epsilon^{-1}\left\langle 1, W(\phi^{n+1}) \right\rangle + \epsilon^{-1} \left\langle 1, W(\phi^{n}) \right\rangle + (2\epsilon)^{-1} W''(\xi^n) \|\phi^{n+1}-\phi^n\|_{L^2}^2 ; \\
 \text{III} =&\ -\gamma \Big\langle (-\Delta)^{-1}(f(\phi^{n})-\omega), f(\phi^{n+1}) - f(\phi^{n}) \Big\rangle + \dfrac{\gamma}{2} \Big\langle (-\Delta)^{-1}(f(\phi^{n})-\omega), f''(\eta^n)(\phi^{n+1} - \phi^{n})^2 \Big\rangle \\
 \quad=&\ -\dfrac{\gamma}{2} \left( \|(-\Delta)^{-\frac{1}{2}}(f(\phi^{n+1})-\omega)\|_{L^{2}}^2 - \|(-\Delta)^{-\frac{1}{2}}(f(\phi^{n})-\omega)\|_{L^{2}}^2 - \|(-\Delta)^{-\frac{1}{2}}(f(\phi^{n+1})-f(\phi^n))\|_{L^{2}}^2 \right) \\
& \ + \dfrac{\gamma}{2} \Big\langle (-\Delta)^{-1}(f(\phi^{n})-\omega), f''(\eta^n)(\phi^{n+1} - \phi^{n})^2\Big\rangle ;\\
 \text{IV} = & -\dfrac{M}{2}\left( \left| \textstyle{\int_{\mathbb{T}^d}} ( f(\phi^{n+1}) - \omega ) \text{d}x \right|^2 -
                                                 \left| \textstyle{\int_{\mathbb{T}^d}} ( f(\phi^{n})     - \omega ) \text{d}x \right|^2 -
                                                 \left| \textstyle{\int_{\mathbb{T}^d}} ( f(\phi^{n+1}) - f(\phi^n)  \text{d}x \right|^2
                                                  \right)& \\
& + \dfrac{M}{2} \left( \textstyle{\int_{\mathbb{T}^d}} ( f(\phi^{n})     - \omega ) \text{d}x\right) f''(\eta^n) \|\phi^{n+1}-\phi^n\|_{L^2}^2      .
\end{align*}
where $\xi^n$ and $\eta^n$ are between $\phi^n$ and $\phi^{n+1}$. Note that the condition (\ref{eqn:ES_condition}) implies  (\ref{eqn:MPP_condition}), Theorem \ref{theorem:MPP_semi} gives $\phi^n,\phi^{n+1} \in [0,1]$, consequently $\xi^n, \eta^n \in (0,1)$. Therefore, we do not need the extension of $f$ as in (\ref{eqn:f_extension}) to perform Taylor expansion above. Finally, inserting the equalities for I--IV back into (\ref{eqn:estimate1}) and noting that $|f'|<L_{f'}, |f''|\le L_{f''}$ and $\int_{\mathbb{T}^d}| f(\phi)-\omega | dx \le  \tilde{\omega} |\mathbb{T}^d|$, it follows that
\begin{align*}
&\dfrac{1}{\tau}\|\phi^{n+1}-\phi^n\|_{L^2}^2 + \frac{\epsilon}{2}\|\nabla\phi^{n+1}-\nabla\phi^n\|_{L^2}^2 + E^{\text{pOK}}[\phi^{n+1}] - E^{\text{pOK}}[\phi^n] \\
=& -\frac{\kappa}{\epsilon}\|\phi^{n+1}-\phi^n\|_{L^2}^2 + \dfrac{W''(\eta^n)}{2\epsilon}\|\phi^{n+1}-\phi^n\|_{L^2}^2  \\
& + \dfrac{\gamma}{2}\|(-\Delta)^{-\frac{1}{2}}(f(\phi^{n+1})-f(\phi^n))\|_{L^{2}}^2 + \dfrac{\gamma}{2} \left\langle (-\Delta)^{-1}(f(\phi^{n})-\omega), f''(\eta^n)(\phi^{n+1} - \phi^{n})^2\right\rangle \\
& + \dfrac{M}{2}\left| \textstyle{\int_{\mathbb{T}^d}} (f(\phi^{n+1}) - f(\phi^n))\text{d}x\right|^2 +  \dfrac{M}{2} \left( \textstyle{\int_{\mathbb{T}^d}} ( f(\phi^n) -\omega )\text{d}x\right) f''(\eta^n) \|\phi^{n+1}-\phi^n\|_{L^2}^2 \\
\le& -\frac{\kappa}{\epsilon}\|\phi^{n+1}-\phi^n\|_{L^2}^2
       + \frac{L_{W}}{2\epsilon}\|\phi^{n+1}-\phi^n\|_{L^2}^2  \\
& + \frac{\gamma}{2}L_{f'}^2 \|(-\Delta)^{-1}\| \|(\phi^{n+1}-\phi^n)\|_{L^2}^2
   + \frac{\gamma}{2} \tilde{\omega}L_{f''}\|(-\Delta)^{-1}\|  \|\phi^{n+1}-\phi^n\|_{L^2}^2\\
& + \frac{M}{2} L_{f'}^2  |\mathbb{T}^d|   \|\phi^{n+1}-\phi^n\|_{L^2}^2
+ \frac{M}{2}  \tilde{\omega} L_{f''} |\mathbb{T}^d|   \|\phi^{n+1}-\phi^n\|_{L^2}^2 \\
= & \left(-\frac{\kappa}{\epsilon} +\frac{L_W}{2\epsilon} + \frac{1}{2}(L_{f'}^2+ \tilde{\omega}L_{f''})\Big(\gamma\|(-\Delta)^{-1}\| + M |\mathbb{T}^d| \Big)  \right)  \|\phi^{n+1}-\phi^n\|_{L^2}^2 \le 0,
\end{align*}
where the last inequality is true given the condition (\ref{eqn:ES_condition}). Consequently it leads to the energy stability.
\end{proof}

\subsection{Error estimate for time-discrete scheme}\label{subsection:errorestimate_timediscrete}

Now we perform an error estimate for the time-discrete scheme (\ref{eqn:pACOK_SemiImplicit}). Assume that the condition (\ref{eqn:ES_condition}) holds (and consequently the MPP condition (\ref{eqn:MPP_condition}) holds), and the initial $\phi^0$ is bounded $0\le\phi^0\le1$ (and consequently $0\le\phi^n\le1$ for any $n\in [\![ N]\!]$).

Subtracting equation (\ref{eqn:pACOK_SemiImplicit}) from the original equation (\ref{eqn:pACOK}) at time $t_{n+1}$ and denoting the error by $\tilde{e}^n = \phi(t_n) - \phi^n$, one has
\begin{align}
&\dfrac{1}{\tau}(\tilde{e}^{n+1} - \tilde{e}^n) - \epsilon \Delta \tilde{e}^{n+1}  \nonumber\\
= &\  R^{n+1} - \dfrac{\kappa}{\epsilon}(\tilde{e}^{n+1}-\tilde{e}^n)
                    + \dfrac{\kappa}{\epsilon}\Big[\phi(t_{n+1})-\phi(t_n)\Big]
                    - \dfrac{1}{\epsilon}\Big[ W'(\phi(t_{n+1})) - W'(\phi^n) \Big] \nonumber \\
  &  - \gamma\Big[ (-\Delta)^{-1}(f(\phi(t_{n+1}))-\omega)f'(\phi(t_{n+1}))
                               - (-\Delta)^{-1}(f(\phi^n)-\omega)f'(\phi^n)\Big]      \nonumber \\
  &  - M\Big[ \left(\textstyle{\int_{\mathbb{T}^d}} (f(\phi(t_{n+1}))-\omega) \text{d}x\right) f'(\phi(t_{n+1}))
                   -\left(\textstyle{\int_{\mathbb{T}^d}} (f(\phi^n)-\omega) \text{d}x\right) f'(\phi^n)\Big] .
\end{align}
where $R^{n+1} = \dfrac{\phi(t_{n+1})-\phi(t_n)}{\tau} -\phi_t(t_{n+1})$ has the following estimate \cite{ShenYang_DCDSA2010}:
\[
\|R^{n+1}\|_{H^s}^2 \le \dfrac{\tau}{3}\int_{t_n}^{t_{n+1}} \|\phi_{tt}(t)\|_{H^s}^2 dt,\quad s = -1, 0.
\]
Taking the $L^2$ inner product with $\tilde{e}^{n+1}$, it follows that
\begin{align}\label{eqn:estimate2}
& \dfrac{1}{2\tau}(\|\tilde{e}^{n+1}\|_{L^2}^2 - \|\tilde{e}^{n}\|_{L^2}^2 + \|\tilde{e}^{n+1}-\tilde{e}^n\|_{L^2}^2) + \epsilon \|\nabla \tilde{e}^{n+1}\|_{L^2}^2 \nonumber\\
= & \ \underbrace{\langle R^{n+1}, \tilde{e}^{n+1} \rangle}_{\text{I}}
         - \dfrac{\kappa}{\epsilon} \left(\dfrac{1}{2}\|\tilde{e}^{n+1}\|_{L^2}^2 - \dfrac{1}{2}\|\tilde{e}^{n}\|_{L^2}^2 + \dfrac{1}{2}\|\tilde{e}^{n+1}-\tilde{e}^n\|_{L^2}^2 \right) \nonumber\\
   &    + \underbrace{\dfrac{\kappa}{\epsilon}\langle\phi(t_{n+1})-\phi(t_n), e^{n+1}\rangle}_{\text{II}}
          \underbrace{ - \dfrac{1}{\epsilon}\langle W'(\phi(t_{n+1})) - W'(\phi^n), e^{n+1}\rangle}_{\text{III}} \nonumber\\
     &  \underbrace{ -  \gamma \Big\langle (-\Delta)^{-1}(f(\phi(t_{n+1}))-\omega)f'(\phi(t_{n+1}))
                               - (-\Delta)^{-1}(f(\phi^n)-\omega)f'(\phi^n), e^{n+1}  \Big\rangle}_{\text{IV}} \nonumber\\
     & \underbrace{ - M \Big\langle  \left( \textstyle{\int_{\mathbb{T}^d}} (f(\phi(t_{n+1}))-\omega) \text{d}x\right) f'(\phi(t_{n+1}))
                   -\left( \textstyle{\int_{\mathbb{T}^d}} (f(\phi^n)-\omega) \text{d}x\right) f'(\phi^n), e^{n+1} \Big\rangle }_{\text{V}}.
\end{align}
For the term I, we have
\begin{align*}
\text{I}
& \le \|R^{n+1}\|_{H^{-1}}\|\tilde{e}^{n+1}\|_{H^1}
 \le \dfrac{1+\frac{|\mathbb{T}^d|}{4\pi^2}}{2\epsilon} \|R^{n+1}\|_{H^{-1}}^2
 + \dfrac{\epsilon}{2\left[1+\frac{|\mathbb{T}^d|}{4\pi^2}\right]}\|\tilde{e}^{n+1}\|_{H^1}^2 \\
&\le \dfrac{1+\frac{|\mathbb{T}^d|}{4\pi^2}}{2\epsilon} \|R^{n+1}\|_{H^{-1}}^2
+ \dfrac{\epsilon}{2}\|\nabla \tilde{e}^{n+1}\|_{L^2}^2
\le \dfrac{1+\frac{|\mathbb{T}^d|}{4\pi^2}}{2\epsilon} \dfrac{\tau}{3}\int_{t_n}^{t_{n+1}}\|\phi_{tt}(t)\|_{H^{-1}}^2 dt
+ \dfrac{\epsilon}{2}\|\nabla \tilde{e}^{n+1}\|_{L^2}^2.
\end{align*}
Note that
\begin{align*}
&\| \phi(t_{n+1}) - \phi(t_n) \|_{L^2}^2 \le \tau \int_{t_n}^{t_{n+1}} \|\phi_t(t)\|_{L^2}^2 dt ; \\
& \|\phi(t_{n+1}) - \phi^n\|_{L^2} \le \|\tilde{e}^{n+1}\|_{L^2} + \|\tilde{e}^{n+1} - \tilde{e}^n\|_{L^2} + \| \phi(t_{n+1}) - \phi(t_n)\|_{L^2} ;
\end{align*}
the terms II and III become
\begin{align*}
\text{II}
\le\ &\dfrac{\kappa}{\epsilon} \| \phi(t_{n+1}) - \phi(t_n) \|_{L^2} \|\tilde{e}^{n+1}\|_{L^2}
 \le \dfrac{\kappa}{\epsilon} \left( \dfrac{1}{2} \| \phi(t_{n+1}) - \phi(t_n) \|^2_{L^2} + \dfrac{1}{2}\|\tilde{e}^{n+1}\|^2_{L^2} \right) \\
 \le\ & \dfrac{\kappa}{\epsilon} \left( \dfrac{\tau}{2} \int_{t_n}^{t_{n+1}} \|\phi_t(t)\|_{L^2}^2 \text{d}t + \dfrac{1}{2}\|\tilde{e}^{n+1}\|^2_{L^2} \right) \\
\text{III}
\le\ & \frac{L_{W''}}{\epsilon}\|\phi(t_{n+1})-\phi^n\|_{L^2}\|\tilde{e}^{n+1}\|_{L^2} \\
\le\ & \frac{L_{W''}}{\epsilon} \left(\|\tilde{e}^{n+1}\|_{L^2}^2 + \|\tilde{e}^{n+1}-\tilde{e}^n\|_{L^2}\|\tilde{e}^{n+1}\|_{L^2} +\| \phi(t_{n+1}) - \phi(t_n) \|_{L^2} \|\tilde{e}^{n+1}\|_{L^2} \right) \\
\le\ & \frac{L_{W''}}{\epsilon} \left(\|\tilde{e}^{n+1}\|_{L^2}^2 +\dfrac{1}{4} \|\tilde{e}^{n+1}-\tilde{e}^n\|_{L^2}^2 + \|\tilde{e}^{n+1}\|_{L^2}^2 +\dfrac{\tau}{2} \int_{t_n}^{t_{n+1}} \|\phi_t(t)\|_{L^2}^2 dt + \dfrac{1}{2} \|\tilde{e}^{n+1}\|_{L^2}^2 \right) \\
\le\ & \frac{L_{W''}}{\epsilon} \left(\dfrac{1}{4} \|\tilde{e}^{n+1}-\tilde{e}^n\|_{L^2}^2  +\dfrac{\tau}{2} \int_{t_n}^{t_{n+1}} \|\phi_t(t)\|_{L^2}^2 dt +  \dfrac{5}{2} \|\tilde{e}^{n+1}\|_{L^2}^2 \right).
\end{align*}
Furthermore, term IV gives
\begin{align*}
\text{IV} =\ &-\gamma \Big( \left\langle (-\Delta)^{-1}(f(\phi(t_{n+1}))-f(\phi^n))f'(\phi(t_{n+1})), \tilde{e}^{n+1} \right\rangle \\
       & + \left\langle (-\Delta)^{-1}(f(\phi^n)-\omega)(f'(\phi(t_{n+1})-f'(\phi^n)), \tilde{e}^{n+1} \right\rangle \Big) \\
\le\  & \gamma L_{f'}^2 \|(-\Delta)^{-1}\| \|\phi(t_{n+1})-\phi^n\|_{L^2}\|\tilde{e}^{n+1}\|_{L^2}   + \gamma \tilde{\omega}L_{f''}\|(-\Delta)^{-1}\| \|\phi(t_{n+1})-\phi^n\|_{L^2}\|\tilde{e}^{n+1}\|_{L^2} \\
\le\ & \gamma(L_{f'}^2+\tilde{\omega}L_{f''})\|(-\Delta)^{-1}\| \left( \dfrac{1}{4} \|\tilde{e}^{n+1}-\tilde{e}^n\|_{L^2}^2  +\dfrac{\tau}{2} \int_{t_n}^{t_{n+1}} \|\phi_t(t)\|_{L^2}^2 dt +  \dfrac{5}{2} \|\tilde{e}^{n+1}\|_{L^2}^2 \right)
\end{align*}
Similar as IV, term V has the estimate
\begin{align*}
\text{V} \le M (L_{f'}^2 + \tilde{\omega}L_{f''} ) |\mathbb{T}^d| \left( \dfrac{1}{4} \|\tilde{e}^{n+1}-\tilde{e}^n\|_{L^2}^2  +\dfrac{\tau}{2} \int_{t_n}^{t_{n+1}} \|\phi_t(t)\|_{L^2}^2 dt +  \dfrac{5}{2} \|\tilde{e}^{n+1}\|_{L^2}^2 \right) .
\end{align*}
Inserting the estimates for terms I-V, then combining all the terms involving $\|\tilde{e}^{n+1}-\tilde{e}^n\|_{L^2}^2$ and $(\tau \int_{t_n}^{t_{n+1}} \|\phi_t(t)\|_{L^2}^2 dt +  5 \|\tilde{e}^{n+1}\|_{L^2}^2)$, and note that $\pm \frac{\kappa}{2\epsilon}\|\tilde{e}^{n+1}\|_{L^2}^2$ are canceled, we have
\begin{align*}
& \dfrac{1}{2\tau}(\|\tilde{e}^{n+1}\|_{L^2}^2 - \|\tilde{e}^{n}\|_{L^2}^2 + \|\tilde{e}^{n+1}-\tilde{e}^n\|_{L^2}^2) + \frac{\epsilon}{2} \|\nabla \tilde{e}^{n+1}\|_{L^2}^2 \\
\le\ & \dfrac{1+\frac{|\mathbb{T}^2|}{4\pi^2}}{2\epsilon} \dfrac{\tau}{3}\int_{t_n}^{t_{n+1}}\|\phi_{tt}(t)\|_{H^{-1}}^2 dt + \dfrac{\kappa}{2\epsilon}\|\tilde{e}^n\|_{L^2}^2 - \dfrac{\kappa}{2\epsilon}\|\tilde{e}^{n+1}-\tilde{e}^n\|_{L^2}^2 + \dfrac{\kappa}{2\epsilon}\tau\int_{t_n}^{t_{n+1}}\|\phi_t(t)\|_{L^2}^2\text{d}t \\
& + \Big(\dfrac{L_{W''}}{\epsilon} + (L_{f'}^2 + \tilde{\omega}L_{f''}) \big(\gamma\|(-\Delta)^{-1}\|+M|\mathbb{T}^d| \big) \Big) \left( \dfrac{1}{4} \|\tilde{e}^{n+1}-\tilde{e}^n\|_{L^2}^2  +\dfrac{\tau}{2} \int_{t_n}^{t_{n+1}} \|\phi_t(t)\|_{L^2}^2 dt +  \dfrac{5}{2} \|\tilde{e}^{n+1}\|_{L^2}^2 \right).
\end{align*}
Provided that the condition (\ref{eqn:ES_condition}) holds, then
\begin{align*}
& \dfrac{1}{2\tau}(\|\tilde{e}^{n+1}\|_{L^2}^2 - \|\tilde{e}^{n}\|_{L^2}^2 + \|\tilde{e}^{n+1}-\tilde{e}^n\|_{L^2}^2) + \frac{\epsilon}{2} \|\nabla \tilde{e}^{n+1}\|_{L^2}^2 \\
\le\ & \dfrac{1+\frac{|\mathbb{T}^2|}{4\pi^2}}{2\epsilon} \dfrac{\tau}{3}\int_{t_n}^{t_{n+1}}\|\phi_{tt}(t)\|_{H^{-1}}^2 dt + \dfrac{\kappa}{2\epsilon}\|\tilde{e}^n\|_{L^2}^2 + \dfrac{\kappa}{2\epsilon}\tau\int_{t_n}^{t_{n+1}}\|\phi_t(t)\|_{L^2}^2\text{d}t  + \dfrac{\kappa}{\epsilon} \left(  \dfrac{\tau}{2} \int_{t_n}^{t_{n+1}} \|\phi_t(t)\|_{L^2}^2 dt +  \dfrac{5}{2} \|\tilde{e}^{n+1}\|_{L^2}^2 \right).\\
= \ &  \dfrac{1+\frac{|\mathbb{T}^2|}{4\pi^2}}{2\epsilon} \dfrac{\tau}{3}\int_{t_n}^{t_{n+1}}\|\phi_{tt}(t)\|_{H^{-1}}^2 dt + \dfrac{\kappa}{\epsilon}\tau\int_{t_n}^{t_{n+1}}\|\phi_t(t)\|_{L^2} + \dfrac{\kappa}{2\epsilon}(\|\tilde{e}^n\|_{L^2}^2 + 5\|\tilde{e}^{n+1}\|_{L^2}^2)
\end{align*}
Dropping the $\|\nabla \tilde{e}^{n+1}\|_{L^2}^2$ term on left hand side, multiplying $2\tau$ on both sides, and summing up the above inequality from 0 to $n-1$, we find
\begin{align*}
\|\tilde{e}^n\|_{L^2}^2 - \|\tilde{e}^0\|_{L^2}^2
\le \dfrac{1+\frac{|\mathbb{T}^2|}{4\pi^2}}{3\epsilon} \tau^2 \|\phi_{tt}\|_{L^2(0,T; H^{-1})}^2
+ 2\tau^2 \dfrac{\kappa}{\epsilon} \|\phi_t\|_{L^2(0,T;L^2)}^2  + \tau \dfrac{\kappa}{\epsilon} \Big( \|\tilde{e}^0\|_{L^2}^2 + 6\sum_{k=1}^{n-1}\|\tilde{e}^k\|_{L^2}^2 + 5\|\tilde{e}^n\|_{L^2}^2 \Big).
\end{align*}
Note that $\tilde{e}^0 = 0$, hence the last inequality becomes
\begin{align*}
(1-5\tau\kappa/\epsilon)\|\tilde{e}^n\|_{L^2}^2
\le \ &\dfrac{1+\frac{|\mathbb{T}^2|}{4\pi^2}}{3\epsilon} \tau^2 \|\phi_{tt}\|_{L^2(0,T; H^{-1})}^2
+ 2\tau^2 \dfrac{\kappa}{\epsilon}  \|\phi_t\|_{L^2(0,T;L^2)}^2   + 6\tau \dfrac{\kappa}{\epsilon}  \sum_{k=1}^{n-1}\|\tilde{e}^k\|_{L^2}^2.
\end{align*}
If $\tau \le \epsilon/(10\kappa)$, we obtain
\begin{align*}
\|\tilde{e}^n\|_{L^2}^2
\le \ &\dfrac{1+\frac{|\mathbb{T}^2|}{4\pi^2}}{3\epsilon} 2\tau^2 \|\phi_{tt}\|_{L^2(0,T; H^{-1})}^2
+ 4\tau^2 \dfrac{\kappa}{\epsilon}  \|\phi_t\|_{L^2(0,T;L^2)}^2   + 12\tau \dfrac{\kappa}{\epsilon}  \sum_{k=1}^{n-1}\|\tilde{e}^k\|_{L^2}^2.
\end{align*}
The discrete Gronwall inequality leads to
\begin{align}
\|\tilde{e}^n\|_{L^2}^2 \le e^{12T\kappa/\epsilon} \left( \dfrac{1+\frac{|\mathbb{T}^2|}{4\pi^2}}{3\epsilon/2}  \|\phi_{tt}\|_{L^2(0,T; H^{-1})}^2 + \dfrac{4\kappa}{\epsilon} \|\phi_t\|_{L^2(0,T;L^2)}^2 \right)\tau^2.
\end{align}

We summarize the above discussion as the following theorem:

\begin{theorem}\label{theorem:error_semidiscrete}
Given $T>0$ and an integer $N>0$ such that $\tau = \frac{T}{N}$ and $t_n = n\tau$ for $n=0,1,\cdots, N$. Assume that $\phi_t \in L^2(0,T; L^2)$ and $\phi_{tt} \in L^2(0,T; H^{-1})$, then for $\kappa$ satisfying the energy stability condition (\ref{eqn:ES_condition}), if the time step size  $\tau \le \epsilon/(10\kappa)$, we have
\begin{align}\label{eqn:errorestimate_time}
\|\phi(t_n) - \phi^n\|_{L^2} \le \tilde{C}\tau, \quad \forall n \in [\![ N]\!]
\end{align}
where $\tilde{C} = e^{6T\kappa/\epsilon} \Big( \frac{1+\frac{|\mathbb{T}^d|}{4\pi^2}}{3\epsilon/2}  \|\phi_{tt}\|_{L^2(0,T; H^{-1})}^2 + \frac{4\kappa}{\epsilon} \|\phi_t\|_{L^2(0,T;L^2)}^2 \Big)^{\frac{1}{2}} $ is a constant independent of $\tau$ and $N$.
\end{theorem}

\section{Fully-discrete Scheme: Maximum Principle Preservation and Energy Stability}

In this section, we propose a fully-discrete scheme by discretizing the spatial operators by a second order finite difference approximation. To this end, we adopt some notations for the finite difference approximation. For the brevity of notations, we will focus the discussion on the 2D case, which can be easily extended to 3D formulation.

\subsection{Second order finite difference scheme for spatial discretizaiton}

We consider $\mathbb{T}^2 = \prod_{i=1}^2 [-X_i,X_i] \subset \mathbb{R}^2$. Let $N_1, N_2$ be positive even integers. Take $h_i = \frac{2X_i}{N_i}, i=1,2$ and $\mathbb{T}^2_h = \mathbb{T}^2\ \cap (\otimes_{i=1}^2 h_i\mathbb{Z})$. We define the index set:
\begin{align*}
S_h &= \left\{ (k_1,k_2)\in\mathbb{Z}^2 | 1\le k_i \le N_i, i=1,2 \right\}.
\end{align*}
Denote by $\mathcal{M}_h$ the collection of periodic grid functions on $\mathbb{T}^2_h$:
\begin{align*}
\mathcal{M}_h = \left\{ f: \mathbb{T}^2_h\rightarrow\mathbb{R} | f_{k_1+m_1N_1, k_2+m_2N_2} =f_{k_1,k_2}, \forall (k_1,k_2)\in S_h, \forall (m_1,m_2)\in \mathbb{Z}^2 \right\}.
\end{align*}
For any $f,g\in\mathcal{M}_h$ and $\textbf{f} = (f^1,f^2)^T, \textbf{g} = (g^1,g^2)^T\in\mathcal{M}_h\times\mathcal{M}_h$, we define the discrete $L^2$ inner product $\langle\cdot,\cdot\rangle_h$, discrete $L^2$ norm $\|\cdot\|_{h,L^2}$ and discrete $L^{\infty}$ norm $\|\cdot\|_{h, L^{\infty}}$ as follows:
\begin{align*}
\langle f, g\rangle_h &= h_x h_y \sum_{(i,j)\in S_h} f_{ij} g_{ij}, \quad \|f\|_{h,L^2} = \sqrt{\langle f, f \rangle_h}, \quad \|f\|_{h, L^{\infty}} = \max_{(i,j)\in S_h} |f_{ij}| ; \\
\langle \textbf{f}, \textbf{g}\rangle_h &= h_x h_y \sum_{(i,j)\in S_h} \left( f_{ij}^1 g_{ij}^1 + f_{ij}^2 g_{ij}^2 \right), \quad  \|\textbf{f}\|_{h,L^2} = \sqrt{\langle \textbf{f}, \textbf{f}\rangle_h}.
\end{align*}
Let $\mathring{\mathcal{M}}_h = \{f\in\mathcal{M}_h | \langle f, 1 \rangle_h = 0\}$ be the collections of all periodic grid functions with zero mean.

We define the second order central difference approximation of the Laplacian operator $\Delta$ as a discrete linear operator $\Delta_h: \mathring{\mathcal{M}}_h \rightarrow \mathring{\mathcal{M}}_h$
\begin{align}
\Delta_h u = f: \Delta_h u_{ij} = \frac{1}{h_1^2}(u_{i-1,j} - 2u_{ij} + u_{i+1,j}) + \frac{1}{h_2^2}(u_{i,j-1} - 2u_{ij} + u_{i,j+1})
\end{align}
where the periodic boundary condition applies when the the indices $i \notin [\![ N_1]\!]$ or $j \notin [\![ N_2]\!]$. Note that $\Delta_h: \mathring{\mathcal{M}}_h \rightarrow \mathring{\mathcal{M}}_h$ is one-to-one, it is safe to define its inverse $(\Delta_h)^{-1}: \mathring{\mathcal{M}}_h \rightarrow \mathring{\mathcal{M}}_h$
\begin{align}
(\Delta_h)^{-1} f = u \quad \text{if and only if} \quad \Delta_h u = f.
\end{align}
We denote by $\|(-\Delta_h)^{-1}\|$ the optimal constant such that $\|(-\Delta_h)^{-1}f\|_{h,L^\infty}\le C \|f\|_{h,L^\infty}$, namely, the norm of the operator $(-\Delta_h)^{-1}$ from $L^{\infty}(\mathring{\mathcal{M}}_h)$ to itself.

Given the discrete Laplacian operator $\Delta_h$ defined above, and denote $\Phi^{n}\approx \phi(x,t_n)|_{\mathbb{T}^2_h}$ the numerical solution, we arrive at the following first order fully-discrete semi-implicit scheme for the pACOK equation (\ref{functional:pOK}): for $\forall n \in [\![ N]\!]$, find $\Phi^{n+1} = (\Phi_{ij}^{n+1}) \in \mathcal{M}_h$ such that
\begin{align}\label{eqn:pACOK_FullDiscrete}
\left(\dfrac{1}{\tau}+\dfrac{\kappa_h}{\epsilon}\right)(\Phi^{n+1}-\Phi^n)  = \epsilon\Delta_h\Phi^{n+1} - \dfrac{1}{\epsilon}W'(\Phi^n) &- \gamma(-\Delta_h)^{-1}(f(\Phi^n)-\omega) \odot f'(\Phi^n) \nonumber \\
&- M \langle f(\Phi^n)-\omega, 1 \rangle_h \text{d}\mathbf{x} f'(\Phi^n),
\end{align}
with $\Phi^0 = (\Phi_{ij}^0) = \phi_0|_{\mathbb{T}^2_h}$ being the given intial data, and $\kappa_h$ the stabilization constant. Here $\text{d}\mathbf{x} = h_1h_2$ and $\odot$ represents pointwise multiplication. The scheme can be reformulated as
\begin{align}\label{eqn:pACOK_FullDiscreteII}
\left(\left(1+\dfrac{\tau\kappa}{\epsilon}\right)I - \tau\epsilon\Delta_h\right)\Phi^{n+1} = \left(1+\dfrac{\tau\kappa}{\epsilon}\right)\Phi^{n}  - \dfrac{\tau}{\epsilon}W'(\Phi^n) & - \tau\gamma(-\Delta_h)^{-1}(f(\Phi^n)-\omega)\odot f'(\Phi^n) \nonumber \\
&- \tau M \langle f(\Phi^n)-\omega, 1 \rangle_h \text{d}\mathbf{x} f'(\Phi^n),
\end{align}
from which the unconditional unique solvability can be guaranteed by realizing the positivity of all the eigenvalues of the operator
$\left( 1+ \tau\kappa\epsilon^{-1} \right) I - \tau \epsilon \Delta_h$ on the left hand side of (\ref{eqn:pACOK_FullDiscreteII}).

\subsection{Maximum principle preservation for fully-discrete scheme}

In this section, we will show that the full-discrete scheme (\ref{eqn:pACOK_FullDiscrete}) is MPP under a condition similar to (\ref{eqn:MPP_condition}). To this end, a discrete counterpart of Lemma \ref{lemma:MPP_condition}  is needed.

\begin{lemma}\label{lemma:MPP_condition2}
Let $\Psi\in\mathcal{M}_h$ be such that $0\le\Psi\le 1$, and define $\mathcal{F}_h: \mathcal{M}_h \rightarrow \mathcal{M}_h$ as follows:
\[
\mathcal{F}_h(\Psi) = \left(1+\dfrac{\tau\kappa_h}{\epsilon}\right)\Psi - \dfrac{\tau}{\epsilon}W'(\Psi) - \tau\gamma(-\Delta_h)^{-1}(f(\psi)-\omega)\odot f'(\Psi) - \tau M \langle f(\psi)-\omega,1 \rangle\emph{d}\mathbf{x} f'(\Psi),
\]
then we have
\[
\max_{0\le\Psi\le 1} \{\mathcal{F}_h(\Psi)\}= 1+\dfrac{\tau\kappa_h}{\epsilon} ; \quad \min_{0\le\Psi\le 1} \{ \mathcal{F}_h(\Psi) \} = 0,
\]
provided that
\begin{align}\label{eqn:MPP_condition2}
\frac{1}{\tau} + \dfrac{\kappa_h}{\epsilon} \ge \dfrac{L_{W''}}{\epsilon} + \tilde{\omega}L_{f''}\Big(\gamma\|(-\Delta_h)^{-1}\| + M|\mathbb{T}^2| \Big).
\end{align}
\end{lemma}

The proof of Lemma \ref{lemma:MPP_condition2} is similar to that of Lemma \ref{lemma:MPP_condition}. The only difference is that the Laplacian operator $-\Delta$ is replaced by a discrete Laplacian operator $-\Delta_h$, and the integral term $\int (f(\psi)-\omega) \text{d}x$ is replaced by the Riemann sum. We therefore omit the details.

Now we present the MPP for the fully-discrete sheme (\ref{eqn:pACOK_FullDiscrete}) or (\ref{eqn:pACOK_FullDiscreteII}).

\begin{theorem}\label{theorem:MPP_fully}
The stabilized fully-discrete scheme (\ref{eqn:pACOK_FullDiscrete}) or (\ref{eqn:pACOK_FullDiscreteII}) is MPP provided that the condition (\ref{eqn:MPP_condition2}) holds.
\end{theorem}
\begin{proof}
Assume that $0\le \Phi^n \le 1$, and $\Phi^{n+1}$ is obtained by the scheme  (\ref{eqn:pACOK_FullDiscrete}). Assume $\Phi^{n+1}$ reaches the maximal value at the index $(i^*,j^*)$, then
\[
(\Delta_h\Phi^{n+1})_{i^*j^*} =  \frac{(\Phi^n)_{i^*-1,j^*}+(\Phi^n)_{i^*+1,j^*}-2(\Phi^n)_{i^*j^*}}{h_1^2} + \frac{(\Phi^n)_{i^*,j^*-1}+(\Phi^n)_{i^*,j^*+1}-2(\Phi^n)_{i^*j^*}}{h_2^2}  \le 0,
\]
and
\[
\left(1+\dfrac{\tau\kappa_h}{\epsilon}\right) (\Phi^{n+1})_{i^*j^*} \le \Big(\mathcal{F}_h((\Phi^n)\Big)_{i^*j^*} \le 1+\dfrac{\tau\kappa}{\epsilon} \Rightarrow \Phi^{n+1} \le 1.
\]
Similarly let $(i_*, j_*)$ be the index for the smallest component of $\Phi^{n+1}$, then
\[
(\Delta_h\Phi^{n+1})_{i_*j_*} =  \frac{(\Phi^n)_{i_*-1,j_*}+(\Phi^n)_{i_*+1,j_*}-2(\Phi^n)_{i_*j_*}}{h_1^2} + \frac{(\Phi^n)_{i_*,j_*-1}+(\Phi^n)_{i_*,j_*+1}-2(\Phi^n)_{i_*j_*}}{h_2^2}  \ge 0,
\]
and
\[
\left(1+\dfrac{\tau\kappa_h}{\epsilon}\right) (\Phi^{n+1})_{i_*j_*} \ge \Big(\mathcal{F}_h((\Phi^n)\Big)_{i_*j_*} \ge 0 \Rightarrow \Phi^{n+1} \ge 0,
\]
which completes the proof.
\end{proof}

\subsection{Energy stability for fully-discrete scheme}

While the stabilized fully-discrete scheme (\ref{eqn:pACOK_FullDiscrete}) is MPP, it is also energy stable for the discrete OK energy functional defined below:
\begin{align}\label{eqn:discreteEnergy}
E_h^{\text{pOK}}[\Phi] = & -\frac{\epsilon}{2} \langle \Delta_h\Phi,\Phi \rangle_h + \frac{1}{\epsilon} \langle W(\Phi),1\rangle_h + \frac{\gamma}{2} \Big\langle (-\Delta_h)^{-1}(f(\Phi)-\omega), (f(\Phi)-\omega) \Big\rangle_h \nonumber \\
& + \frac{M}{2} \Big( \langle f(\Phi^n)-\omega, 1 \rangle_h \text{d}\mathbf{x} \Big)^2.
\end{align}

\begin{theorem}
Assume the initial $\Phi^0$ satisfies $0\le \Phi^0\le 1$, then the stabilized fully-discrete semi-implicit scheme  (\ref{eqn:pACOK_FullDiscrete}) or (\ref{eqn:pACOK_FullDiscreteII}) is unconditionally energy stable in the sense that
\begin{align}
E_h^{\emph{pOK}}[\Phi^{n+1}] \le E_h^{\emph{pOK}}[\Phi^{n}]
\end{align}
 provided that
\begin{align}\label{eqn:discreteES_condition}
\frac{\kappa_h}{\epsilon} \ge \frac{L_{W''}}{\epsilon} + (L_{f'}^2+\tilde{\omega}L_{f''})\Big(\gamma\|(-\Delta_h)^{-1}\| + M |\mathbb{T}^2|\Big).
\end{align}
\end{theorem}
\begin{proof}
The proof is similar to that of Theorem \ref{theorem:EnergyStability}. To see how the discrete operators apply in the proof, we will still show it in details. Taking the discrete $L^2$ inner product with $\Phi^{n+1}-\Phi^{n}$ on the two sides of (\ref{eqn:pACOK_FullDiscrete}), we have
\begin{align}\label{eqn:estimate2}
&\dfrac{1}{\tau} \|\Phi^{n+1} - \Phi^{n}\|_{h, L^2}^2 \nonumber\\
= &\  - \dfrac{\kappa_h}{\epsilon}\|\Phi^{n+1} - \Phi^{n}\|_{h,L^2}^2   +\underbrace{\epsilon \langle \Delta_h\Phi^{n+1}, \Phi^{n+1} - \Phi^{n} \rangle_h }_{\text{I}} \underbrace{ - \epsilon^{-1} \langle W'(\Phi^{n}), \Phi^{n+1} - \Phi^{n}\rangle_h }_{\text{II}} \nonumber\\
& \underbrace{ -\gamma \left\langle (-\Delta_h)^{-1}(f(\Phi^n)-\omega)f'(\Phi^n), \Phi^{n+1}-\Phi^{n} \right\rangle_h }_{\text{III}}  \underbrace{ -M \langle f(\Phi^n)-\omega, 1\rangle\text{d}\textbf{x} \left\langle f'(\Phi^n), \Phi^{n+1}-\Phi^{n} \right\rangle_h }_{\text{IV}}.
\end{align}
Using the identity $a\cdot (a-b) = \frac{1}{2}|a|^2 - \frac{1}{2}|b|^2 + \frac{1}{2}|a-b|^2$ and $b\cdot (a-b) = \frac{1}{2}|a|^2 - \frac{1}{2}|b|^2 - \frac{1}{2}|a-b|^2$, we have:
\begin{align*}
\text{I} =&\ \frac{\epsilon}{2} \Big( \Big\langle \Delta_h\Phi^{n+1}, \Phi^{n+1} \Big\rangle_h - \Big\langle \Delta_h\Phi^{n}, \Phi^{n} \Big\rangle_h + \Big\langle \Delta_h(\Phi^{n+1}-\Phi^{n}), \Phi^{n+1}-\Phi^n \Big\rangle_h \Big); \\
\text{II} =& -\epsilon^{-1}\left\langle 1, W(\Phi^{n+1}) \right\rangle_h + \epsilon^{-1} \left\langle 1, W(\Phi^{n}) \right\rangle_h + (2\epsilon)^{-1} W''(\xi^n) \|\Phi^{n+1}-\Phi^n\|_{h, L^2}^2 ; \\
 \text{III} =&\ -\gamma \Big\langle (-\Delta_h)^{-1}(f(\Phi^{n})-\omega), f(\Phi^{n+1}) - f(\Phi^{n})\Big\rangle_h + \dfrac{\gamma}{2} \Big\langle (-\Delta_h)^{-1}(f(\Phi^{n})-\omega), f''(\eta^n)(\Phi^{n+1} - \Phi^{n})^2\Big\rangle_h \\
 \quad=&\ -\dfrac{\gamma}{2} \Big( \Big\langle (-\Delta_h)^{-1}(f(\Phi^{n+1})-\omega), f(\Phi^{n+1})-\omega \Big\rangle_{h} - \Big\langle (-\Delta_h)^{-1}(f(\Phi^{n})-\omega), f(\Phi^{n})-\omega \Big\rangle_{h}  \\
& \hspace{0.4in} - \Big\langle (-\Delta_h)^{-1}(f(\Phi^{n+1})-f(\Phi^{n})),f(\Phi^{n+1})-f(\Phi^{n}) \Big\rangle_{h} \Big)   \\
& \ + \dfrac{\gamma}{2} \Big\langle (-\Delta_h)^{-1}(f(\Phi^{n})-\omega), f''(\eta^n)(\Phi^{n+1} - \Phi^{n})^2\Big\rangle_h ;\\
 \text{IV} = & -\dfrac{M}{2}\Big( \left(\langle f(\Phi^{n+1}) - \omega, 1 \rangle_h  \text{d}\textbf{x} \right)^2 -
                                                 \left(\langle f(\Phi^{n}) - \omega, 1 \rangle_h  \text{d}\textbf{x} \right)^2 -
                                                 \left(\langle f(\Phi^{n+1}) - f(\Phi^{n}), 1 \rangle_h  \text{d}\textbf{x} \right)^2
                                                  \Big)& \\
& + \dfrac{M}{2} \langle f(\Phi^{n}) - \omega, 1 \rangle_h  \text{d}\textbf{x} f''(\eta^n) \|\Phi^{n+1}-\Phi^n\|_{h,L^2}^2,
\end{align*}
where $\xi^n$ and $\eta^n$ are between $\Phi^n$ and $\Phi^{n+1}$ due to the smoothness up to 2nd order derivative for $f$ and $W$. Note that the condition (\ref{eqn:discreteES_condition}) implies  (\ref{eqn:MPP_condition2}), owing to Theorem \ref{theorem:MPP_fully}, $\Phi^n,\Phi^{n+1} \in [0,1]$, therefore $\xi^n, \eta^n \in (0,1)$. Finally inserting the equalities for I--IV back into (\ref{eqn:estimate2}) and noting that $|f'|<L_{f'}, |f''|\le L_{f''}$ and $|\langle f(\Phi^{n})-\omega, 1\rangle_h  \text{d}\textbf{x}| \le \tilde{\omega}|\mathbb{T}^d|$, it follows that
\begin{align*}
&\dfrac{1}{\tau}\|\Phi^{n+1}-\Phi^n\|_{h,L^2}^2 + \frac{\epsilon}{2}\Big\langle -\Delta_h(\Phi^{n+1}-\Phi^{n}), \Phi^{n+1}-\Phi^n\Big\rangle_h + E_h^{\text{pOK}}[\Phi^{n+1}] - E_h^{\text{pOK}}[\Phi^n] \\
=& -\frac{\kappa_h}{\epsilon}\|\Phi^{n+1}-\Phi^n\|_{h,L^2}^2 + \dfrac{W''(\eta^n)}{2\epsilon}\|\Phi^{n+1}-\Phi^n\|_{h,L^2}^2  \\
& + \dfrac{\gamma}{2} \Big\langle (-\Delta_h)^{-1}(f(\Phi^{n+1})-f(\Phi^n)), f(\Phi^{n+1}-f(\Phi^n)) \Big\rangle_h + \dfrac{\gamma}{2} \Big\langle (-\Delta_h)^{-1}(f(\Phi^{n})-\omega), f''(\eta^n)(\Phi^{n+1} - \Phi^{n})^2\Big\rangle_h \\
& + \dfrac{M}{2} \Big( \langle f(\Phi^{n+1}) - f(\Phi^n), 1 \rangle_h \text{d}\textbf{x} \Big)^2 +  \dfrac{M}{2}  \langle f(\Phi^n) -\mathbb{T}^d, 1 \rangle_h \text{d}\textbf{x}  f''(\eta^n) \|\Phi^{n+1}-\Phi^n\|_{h, L^2}^2 \\
\le& -\frac{\kappa_h}{\epsilon}\|\Phi^{n+1}-\Phi^n\|_{h,L^2}^2
       + \frac{L_{W}}{2\epsilon}\|\Phi^{n+1}-\Phi^n\|_{h,L^2}^2  \\
& + \frac{\gamma}{2}L_{f'}^2 \|(-\Delta_h)^{-1}\| \|(\Phi^{n+1}-\Phi^n)\|_{h,L^2}^2
   + \frac{\gamma}{2}\tilde{\omega}L_{f''}\|(-\Delta_h)^{-1}\| \|\Phi^{n+1}-\Phi^n\|_{h,L^2}^2\\
& + \frac{M}{2} L_{f'}^2  |\mathbb{T}^2|   \|\Phi^{n+1}-\Phi^n\|_{h,L^2}^2
+ \frac{M}{2} \tilde{\omega}L_{f''} |\mathbb{T}^2|   \|\Phi^{n+1}-\Phi^n\|_{h,L^2}^2 \\
= & \left(-\frac{\kappa_h}{\epsilon} +\frac{L_{W''}}{2\epsilon} + \frac{1}{2}(L_{f'}^2+\tilde{\omega}L_{f''})\left(\gamma\|(-\Delta)^{-1}\| + M |\mathbb{T}^2|\right)  \right)  \|\Phi^{n+1}-\Phi^n\|_{h,L^2}^2 \le 0,
\end{align*}
where the last inequality is due to the condition (\ref{eqn:discreteES_condition}). Consequently it leads to the energy stability.
\end{proof}

\subsection{Error estimate for fully-discrete scheme}

Now we perform the error estimate for the fully-discrete scheme (\ref{eqn:pACOK_FullDiscrete}). To begin with, we assume that the condition (\ref{eqn:discreteES_condition}) for the discrete energy stability holds (and therefore the discrete MPP condition (\ref{eqn:MPP_condition2}) holds). We assume that the initial data $\Phi^0$ is bounded $0\le \Phi^0 \le 1$ (and therefore $0\le \Phi^n \le 1$ for $\forall n \in [\![ N]\!]$).

For the rest of this section, we simply denote by $\phi(t_n)$ the true solution $\phi(x,t_n)$ limited on $\mathbb{T}^2$. Then $\phi(t_n)$ solves the following discrete equation:
\begin{align*}\label{eqn:LTE}
\Big( \frac{1}{\tau} + \frac{\kappa_h}{\epsilon} \Big) (\phi(t_{n+1})-\phi(t_n)) = &\ \epsilon \Delta_h \phi(t_{n+1}) - \frac{1}{\epsilon} W'(\phi(t_n))  \\
&- \gamma(-\Delta_h)^{-1}(f(\phi(t_n))-\omega)\odot f'(\phi(t_n)) \\
& - M  \langle f(\Phi^n)-\omega, 1 \rangle_h \text{d}\mathbf{x} f'(\phi(t_n)) + \Gamma^{n+1}
\end{align*}
where $\Gamma^{n+1}$ is the local truncation error and satisfies:
\begin{align}
\|\Gamma^{n+1}\|_{h,L^{\infty}} \le C_1(\tau + h_1^2 + h_2^2)
\end{align}
for some $C_1\ge 0$ depending only on $\phi, T, |\mathbb{T}^2|$ but not on $\tau, h_1$ and $h_2$. More precisely, Taylor expansion results in the estimate
\begin{align*}
\|\Gamma^{n+1}\|_{h,L^{\infty}} \le C\Big( \tau\|\phi_{tt}\|_{L^{\infty}(0,T;L^{\infty})} + (h_1^2 + h_2^2)\|\phi\|_{L^{\infty}(0,T;C^4)}   \Big),
\end{align*}
therefore
\begin{align*}
C_1 \le  \|\phi_{tt}\|_{L^{\infty}(0,T;L^{\infty})} + \|\phi\|_{L^{\infty}(0,T;C^4)}.
\end{align*}
Substract (\ref{eqn:pACOK_FullDiscreteII}) from (\ref{eqn:LTE}), and let $e^n = \phi(t_{n}) - \Phi^n$, one has:
\begin{align*}
\Big( \frac{1}{\tau} + \frac{\kappa_h}{\epsilon} \Big) (e^{n+1}-e^{n}) = &\ \epsilon \Delta_h e^{n+1}  -\epsilon^{-1} \Big(W'(\phi(t_n))-W'(\Phi^n)\Big) \\
&  -\gamma\Big( (-\Delta_h)^{-1}(f(\phi(t_n))-\omega)\odot f'(\phi(t_n)) - (-\Delta_h)^{-1}(f(\Phi^n)-\omega)\odot f'(\Phi^n) \Big)  \\
& -M \Big( \langle f(\phi(t_n))-\omega, 1 \rangle_h \text{d}\mathbf{x} f'(\phi(t_n)) - \langle f(\Phi^n)-\omega, 1 \rangle_h \text{d}\mathbf{x} f'(\Phi^n) \Big) + \Gamma^{n+1}.
\end{align*}
Taking the discrete $L^2$ inner product by $e^{n+1}$ on the two sides yields
\begin{align*}
&\underbrace{\Big( \frac{1}{\tau} + \frac{\kappa_h}{\epsilon} \Big) \langle e^{n+1}-e^{n}, e^{n+1} \rangle_h}_{\text{I}} \\
= &\ \epsilon \langle \Delta_h e^{n+1}, e^{n+1} \rangle_h \underbrace{ -\epsilon^{-1} \Big\langle W'(\phi(t_n))-W'(\Phi^n), e^{n+1} \Big\rangle_h }_{\text{II}} \\
& \underbrace{ -\gamma \Big\langle (-\Delta_h)^{-1}(f(\phi(t_n))-\omega)\odot f'(\phi(t_n)) - (-\Delta_h)^{-1}(f(\Phi^n)-\omega)\odot f'(\Phi^n), e^{n+1} \Big\rangle_h }_{\text{III}} \\
&\underbrace{ -M \Big\langle \langle f(\phi(t_n))-\omega, 1 \rangle_h \text{d}\mathbf{x} f'(\phi(t_n)) -\langle f(\Phi^n)-\omega, 1 \rangle_h \text{d}\mathbf{x} f'(\Phi^n), e^{n+1} \Big\rangle_h}_{\text{IV}}+ \langle \Gamma^{n+1}, e^{n+1} \rangle.
\end{align*}
Terms I and II become
\begin{align*}
\text{I} &= \frac{1}{2}\Big( \frac{1}{\tau} + \frac{\kappa_h}{\epsilon} \Big)\Big( \|e^{n+1}\|_{h, L^2}^2  -  \|e^{n}\|_{h, L^2}^2 +  \|e^{n+1} - e^n\|_{h, L^2}^2  \Big); \\
\text{II} &\le L_{W''}(2\epsilon)^{-1}\Big( \|e^n\|_{h,L^2}^2 + \|e^{n+1}\|_{h,L^2}^2 \Big);
\end{align*}
Furthermore, term III gives
\begin{align*}
\text{III} = \ &-\gamma \Big( \Big\langle (-\Delta_h)^{-1}(f(\phi(t_{n}))-f(\Phi^n))\odot f'(\phi(t_{n})), e^{n+1} \Big\rangle \\
       & + \Big\langle (-\Delta_h)^{-1}(f(\Phi^n)-\omega)\odot (f'(\phi(t_{n})-f'(\Phi^n)), e^{n+1} \Big\rangle \Big) \\
\le\  & \gamma L_{f'}^2 \|(-\Delta_h)^{-1}\| \|e^n\|_{h, L^2}\|e^{n+1}\|_{h, L^2}   + \gamma \tilde{\omega}L_{f''}\|(-\Delta_h)^{-1}\| \|e^n\|_{h, L^2}\|e^{n+1}\|_{h, L^2} \\
\le\ & \dfrac{1}{2}\gamma(L_{f'}^2+\tilde{\omega}L_{f''})\|(-\Delta_h)^{-1}\| \Big( \|e^n\|_{h,L^2}^2 + \|e^{n+1}\|_{h,L^2}^2 \Big).
\end{align*}
Similar as III, term IV has the estimate
\begin{align*}
\text{V} = \dfrac{1}{2} M (L_{f'}^2+\tilde{\omega}L_{f''}) |\mathbb{T}^2| \Big( \|e^n\|_{h,L^2}^2 + \|e^{n+1}\|_{h,L^2}^2 \Big)
\end{align*}
Inserting the estimates for terms I-IV and combine all the terms, we have
\begin{align*}
& \dfrac{1}{2\tau}(\|e^{n+1}\|_{h,L^2}^2 - \|e^{n}\|_{h,L^2}^2 + \|e^{n+1}-e^n\|_{h,L^2}^2) + \epsilon \langle -\Delta_h e^{n+1}, e^{n+1} \rangle_h \\
\le\ & - \dfrac{\kappa_h}{2\epsilon}(\|e^{n+1}\|_{h,L^2}^2 - \|e^{n}\|_{h,L^2}^2 + \|e^{n+1}-e^n\|_{h,L^2}^2)  \\
& + \frac{1}{2}\Big(\dfrac{L_{W''}}{\epsilon} + (L_{f'}^2 + \tilde{\omega}L_{f''}) \big(\gamma\|(-\Delta_h)^{-1}\|+M|\mathbb{T}^2| \big) \Big) \Big( \|e^n\|_{h,L^2}^2 + \|e^{n+1}\|_{h,L^2}^2 \Big) \\
& + \frac{1}{2}\Big(\|\Gamma^{n+1}\|_{h,L^2}^2 + \|e^{n+1}\|_{h,L^2}\Big).
\end{align*}
Owing to the condition (\ref{eqn:discreteES_condition}), the estimate further becomes
\begin{align*}
& \dfrac{1}{2\tau}(\|e^{n+1}\|_{h,L^2}^2 - \|e^{n}\|_{h,L^2}^2 + \|e^{n+1}-e^n\|_{h,L^2}^2) + \epsilon \langle -\Delta_h e^{n+1}, e^{n+1} \rangle_h \\
\le\ & - \dfrac{\kappa_h}{2\epsilon}\Big(\|e^{n+1}\|_{h,L^2}^2 - \|e^{n}\|_{h,L^2}^2 + \|e^{n+1}-e^n\|_{h,L^2}^2\Big) + \frac{\kappa_h}{2\epsilon}\Big( \|e^n\|_{h,L^2}^2 + \|e^{n+1}\|_{h,L^2}^2 \Big) \\
& + \frac{1}{2}\Big(\|\Gamma^{n+1}\|_{h,L^2}^2 + \|e^{n+1}\|_{h,L^2}\Big) \\
=\  & \frac{\kappa_h}{\epsilon} \|e^n\|_{h,L^2}^2 - \frac{\kappa_h}{2\epsilon} \|e^{n+1}-e^n\|_{h,L^2}^2 + \frac{1}{2}\Big(\|\Gamma^{n+1}\|_{h,L^2}^2 + \|e^{n+1}\|_{h,L^2}\Big) .
\end{align*}
Multiplying $2\tau$ on two sides, dropping the term involving $\|e^{n+1}-e^n\|_{h,L^2}^2$, and note that $\langle -\Delta_h e^{n+1}, e^{n+1} \rangle_h \ge 0$, the above inequality becomes
\begin{align*}
\|e^{n+1}\|_{h,L^2}^2 - \|e^{n}\|_{h,L^2}^2  \le \frac{2\kappa_h}{\epsilon} \tau \|e^n\|_{h,L^2}^2 +  \tau \|e^{n+1}\|_{h,L^2}  + \tau \|\Gamma^{n+1}\|_{h,L^2}^2  .
\end{align*}
Summing over $n$ and use $e^0 = 0$, we obtain
\begin{align*}
\|e^n\|_{h,L^2}^2 & \le \frac{2\kappa_h}{\epsilon}\tau \sum_{j=0}^{n-1} \|e^j\|_{h,L^2}^2 + \tau \sum_{j=1}^{n} \|e^j\|_{h,L^2}^2 + n\tau \|\Gamma^{n+1}\|_{h,L^2}^2 \\
& \le \tau \|e^n\|_{h,L^2}^2 + \left(1+\frac{2\kappa_h}{\epsilon}\right) \tau \sum_{j=1}^{n-1}\|e^j\|_{h,L^2}^2 + T \|\Gamma^{n+1}\|_{h,L^2}^2
\end{align*}
If the step size $\tau$ is sufficiently small, say $\tau \le 1/2$, then the above inequality becomes
\begin{align*}
\|e^n\|_{h,L^2}^2  \le  \left(1+\frac{2\kappa_h}{\epsilon}\right) 2\tau \sum_{j=1}^{n-1}\|e^j\|_{h,L^2}^2 + 2T \|\Gamma^{n+1}\|_{h,L^2}^2,
\end{align*}
which lead to
\begin{align*}
\|e^n\|_{h,L^2}^2 \le e^{(1+2\kappa_h\epsilon^{-1})2\tau(n-1)}2T\|\Gamma^{n+1}\|_{h,L^2}^2 \le e^{2T(1+2\kappa_h\epsilon^{-1})}2T |\mathbb{T}^2| C_1^2(\tau+h_1^2 + h_2^2)^2 = C^2(\tau+h_1^2 + h_2^2)^2
\end{align*}
owing to the Gronwall inequality and the fact that $\|\Gamma^{n+1}\|_{h,L^2}^2 \le |\mathbb{T}^2| \|\Gamma^{n+1}\|_{h,L^{\infty}}^2$.

Summarizing the above discussion lead to the following theorem:
\begin{theorem}
Given $T>0$ and an integer $N>0$ such that $\tau = T/N$ and $t_n = n\tau$ for $n=0,1,\cdots,N$. Assume the initial value $\phi_0$ is smooth, periodic and bounded $0\le \phi_0 \le 1$, and the exact solution $\phi(x,t)$ is sufficiently smooth. Let the stabilization constant $\kappa_h$ satisfy the condition (\ref{eqn:discreteES_condition}). We denote by $\{\Phi^n\}_{n=1}^N = \{(\Phi_{ij}^n)\}_{n=1}^N$ the approximate solution calculated by the scheme (\ref{eqn:pACOK_FullDiscrete}) with $\Phi^0 = \phi_0|_{\mathbb{T}^2}$. If the step size $\tau$ is sufficiently small, we have
\begin{align}
\|\phi(t_n) - \Phi^{n}\|_{h,L^2} \le C (\tau + h_1^2 + h_2^2), \quad n \in [\![ N]\!],
\end{align}
where $C>0$ is some generic constant which depends on $\phi, T, \kappa_h, \epsilon, \gamma, M, |\mathbb{T}^2|$ but is independent of $\tau, h_1, h_2$.
\end{theorem}

\section{MPP schemes for a general Allen-Cahn type model}\label{section:GeneralFramework}

Our study on MPP can be extended into a more general setting. Consider a general Allen-Cahn type dynamics:
\begin{align}\label{eqn:generalAC}
\dfrac{\partial}{\partial t} \phi = \epsilon\Delta\phi - \dfrac{1}{\epsilon}W'(\phi) - \mathcal{L}(f(\phi)-\omega)f'(\phi) - M\int_{\mathbb{T}^d} (f(\phi)-\omega)\ \text{d}x\cdot f'(\phi),
\end{align}
where $\mathcal{L}$ is a positive semi-definite linear operator from $L^{\infty}(\mathbb{T}^d)$ to $L^{\infty}(\mathbb{T}^d)$ with the norm denoted by $\|\mathcal{L}\|$. The last term on the right hand size counts for a possible volume constraint (\ref{eqn:Volume}) when necessary. For instance, if $\mathcal{L} = (-\Delta)^{-1}$,  then the volume constraint (\ref{eqn:Volume}) is necessary, and it recovers the pACOK dynamics (\ref{eqn:pACOK}). If $\mathcal{L} = (I - \gamma^2\Delta)^{-1}$, then the volume constraint is unnecessary, and we set $M = 0$. This general dynamics can be viewed as the $L^2$ gradient flow dynamics associated to the free energy functional
\begin{align}\label{eqn:generalframework}
E^{\text{ge}}[\phi] = \int_{\mathbb{T}^d} \frac{\epsilon}{2}|\nabla\phi|^2 + \frac{1}{\epsilon}W(\phi)\ \text{d}x + \int_{\mathbb{T}^d} |\mathcal{L}^{\frac{1}{2}}(f(\phi)-\omega)|^2\ \text{d}x + \frac{M}{2}\Big( \int_{\mathbb{T}^d} f(\phi)-\omega\ \text{d}x \Big)^2,
\end{align}
where depending on the different form of the operator $\mathcal{L}$, we might or might not need the volume constraint.

Here are some examples which fit into the general framework described above.
\begin{itemize}
\item In the micromagnetic model for garnet films \cite{CondetteMelcherSuli_MathComp2010} with $d = 2$, the operator $\mathcal{L}$ is being characterized by its eigenvalues $\lambda(k) = \frac{1-\exp(-\delta|k|)}{\delta|k|}$. Here $\delta>0$ corresponds to the relative film thickness. There is no volume constraint in this model.

\item A nonlocal geometric variational problem studied by \cite{RenTruskinovsky_Elasticity2000} takes $\mathcal{L} = (I - \gamma^2\Delta)^{-1}$ with no volume constraint. This problem can lead to the FitzHugh-Nagumo system \cite{ChenChoiHuRen_SIMA2018}.

\item We can consider a positive semi-define linear operator $\mathcal{L}$ which is the inverse of the following nonlocal operator $\mathcal{K}: L^2(\mathbb{T}^d) \rightarrow L^2(\mathbb{T}^d)$
\begin{align*}
\mathcal{K}: v(x) \longmapsto \int_{\mathbb{T}^d} K(x-y) (v(x) - v(y))\ \text{d}y,
\end{align*}
in which the kernel $K$ is nonnegative, radial, $\mathbb{T}^d$-periodic with bounded second moment \cite{Du_Book2020}. This can be viewed as a nonlocal OK model for the diblock copolymer system.

\item One example that cannot fit into the general framework (\ref{eqn:generalframework}) but still satisfy the MPP property is  the phase field variational implicit solvation model (pVISM), in which the free energy is formulated as \cite{Zhao_2018CMS}:
\[
E^{\text{pVISM}}[\phi] = \int_{\mathbb{T}^d} \frac{\epsilon}{2}|\nabla\phi|^2 + \frac{1}{\epsilon}W(\phi)\ \text{d}x + \int_{\mathbb{T}^d} f(\phi(x)) U(x;X)\ \text{d}x.
\]
Here $X = (x_1,\cdots, x_m)$ are the locations of the $m$ solute atoms, and $U$ is the potential between the solute atoms $X$ and solvent molecules $x$ (for instance, water). The phase field $\phi$ labels the solvent  so that the nonlocal interaction by $U$ takes integral only in the solvent region. The potential $U$ in pVISM typically consists of two parts, the solute-solvent van der Waals interaction and the electrostatic interaction.  Additionally, the potential $U$ is cut off as a constant near the solute atoms $X$ so that it remains bounded. See \cite{Zhao_2018CMS} and the references therein for the detailed discussion. In the next section, we will use pVISM as an example to show that choosing $f(\phi) = 3\phi^2 - 2\phi^3$ make the MPP while $f(\phi) = \phi$ violates the MPP. See Subsection \ref{subsection:pVISM} for the details.
\end{itemize}

In this general setting, the $L^2$ gradient flow dynamics always hold the MPP as in the following theorem.
\begin{theorem}\label{theorem:generalMPP_continuous}
The general $L^2$ gradient flow dynamics (\ref{eqn:generalAC}) is maximum principle preserving, namely, if $0\le \phi_0 \le 1$, then $0\le \phi(t) \le 1$ for any $t>0$, provided that
\begin{align}\label{eqn_MPPcondition_continuous_generalsetting}
\frac{\epsilon\tilde{\omega}}{6}\Big[ \|\mathcal{L}\| + \tilde{M} |\mathbb{T}^d|  \Big] \le 1,
\end{align}
where $\tilde{M} = M$ if there is a volume constraint (\ref{eqn:Volume}), and $\tilde{M} = 0$ if there is no volume constraint.
\end{theorem}
The proof is identical to that of theorem \ref{theorem:MPP_continuous} by replacing $\gamma(-\Delta)^{-1}$ by $\mathcal{L}$, so we omit it.

\section{Numerical simulations}

In this section, some numerical examples will be presented to validate the proposed schemes. Moreover some interesting patterns arising from the OK model will be shown. To begin with, let us briefly explain how to implement the numerical scheme (\ref{eqn:pACOK_FullDiscreteII}). The implementation is as follows:

\begin{enumerate}
\item At the $n$-th step, take the Discrete Fourier Transform(DFT) on the right hand side $\widehat{RHS}_{jk}$;
\item Calculate the DFT of $\Phi^{n+1}$ as $\hat{\Phi}^{n+1}_{jk} = \frac{\widehat{RHS}_{jk}}{\left(1+\frac{\tau\kappa}{\epsilon}\right) + \frac{4}{h_1^2}\sin^2\left(\frac{j\pi h_1}{2X_1}\right) + \frac{4}{h_2^2}\sin^2\left(\frac{k\pi h_2}{2X_2}\right) }$;
\item Take the inverse DFT of $\hat{\Phi}^{n+1}_{jk}$ to obtain $\Phi^{n+1}$ and move to the $(n+1)$-th step;
\end{enumerate}

Now we solve the pACOK equation (\ref{eqn:pACOK_FullDiscreteII}) coupled with periodic boundary condition. In this section, we fix $\mathbb{T}^2 = [-1,1)^2 \subset \mathbb{R}^2$ and $N = N_1 = N_2 = 256$ unless stated otherwise. We set the stopping criteria for the time iteration by:
\begin{align}
\dfrac{\|\Phi^{n+1}-\Phi^{n}\|_{h,L^{\infty}}}{\tau} \le \text{TOL} = 10^{-3}.
\end{align}
The penalty constant is taken to be sufficiently large $M \gg 1$. We take a sufficiently large value of $\kappa_h>0$ to fulfill the energy stability condition (\ref{eqn:discreteES_condition}) (and therefore fulfill the MPP condition (\ref{eqn:MPP_condition2})), say $\kappa_h = 2000$. Other parameters such as $\epsilon, \gamma, \tau, \omega$ might vary for different simulations.

\subsection{Rate of convergence}

We first of all test the convergence rates and the spatial accuracy of the scheme (\ref{eqn:pACOK_FullDiscreteII}). For this numerical experiment, we fix $\omega = 0.1$, and take a round disk as the initial data $\Phi^0 = 0.5 + 0.5\tanh(\frac{r_0-r}{\epsilon/3})$ with $r_0 = \sqrt{\omega|\mathbb{T}^2|/\pi}+0.1$.  The simulation is performed until $T = 0.02$. For the rate of convergence, we take the solution generated by the scheme (\ref{eqn:pACOK_FullDiscreteII}) with $\tau = 10^{-6}$ and $N=2^8$ (consequently $h = h_1 = h_2 = \frac{2X_1}{N}=\frac{1}{128}$) as the benchmark solution. Then we take several values of step size larger than $\tau = 10^{-6}$, each is the half of the previous one, and compute the discrete $L^2$ error between the numerical solutions with larger step sizes and the benchmark one. Table \ref{table:convergence_rates} presents the errors and the convergence rates based on the data at $T = 0.02$ for the scheme (\ref{eqn:pACOK_FullDiscreteII}) with time step sizes being halved from $\tau = 10^{-4}$ to $10^{-4}/16$. We test the convergence rates for three different values of $\epsilon = 5h, 10h $ and $20h$. $\gamma = 100$ is fixed. We can see from the table that the numerically computed convergence rates all tend to approach the theoretical value 1.

\label{table:convergence_rates}
\begin{table}[H]
\begin{center}
\begin{tabular}{ccccccc}

   & \multicolumn{2}{c}{$\epsilon = 5h$} &\multicolumn{2}{c}{$\epsilon = 10h$} &\multicolumn{2}{c}{$\epsilon = 20h$}  \\ \hline
$\tau $ & Error & Rate & Error & Rate & Error & Rate \\ \hline
1e-4/$2^0$ & 1.936e-1 & --- & 1.555e-1 & --- & 5.902e-2 & --- \\

1e-4/$2^1$ & 1.542e-1& 0.33 & 9.465e-2 & 0.72 & 2.376e-2 & 1.31
	 	 \\
1e-4/$2^2$ & 1.076e-1 & 0.52 & 4.858e-2& 0.96 & 9.233e-3 & 1.36
	 	 \\
1e-4/$2^3$ & 6.423e-2 & 0.74 & 2.247e-2& 1.11 & 3.752e-3 & 1.29
		 \\
1e-4/$2^4$ & 3.270e-2 & 0.97 & 9.787e-3 & 1.20 & 1.556e-3  & 1.27
	 	 \\ 	
1e-6 (BM)  & --- & ---& --- & ---& --- & ---\\ \hline
\end{tabular}
\end{center}
\caption{The errors and the corresponding convergence rates at time $T=0.02$ by the scheme (\ref{eqn:pACOK_FullDiscreteII}) for different values of $\epsilon$. In this simulation, $\omega = 0.1, \gamma = 100, M = 1000, \kappa_h = 2000, N = 256$.}
\end{table}

\subsection{Comparison between $f(\phi) = 3\phi^2 - 2\phi^3$ and $f(\phi) = \phi$ regarding to MPP} \label{subsection:pVISM}

In this section, we show an example to see the effect of $f(\phi) = 3\phi^2 - 2\phi^3$ on MPP. We consider the 1D pVISM system with $f(\phi) = 3\phi^2 - 2\phi^3$, which holds the MPP theoretically, and with $f(\phi) = \phi$, the traditional choice which might lose the MPP. The parameters are taken as $L_x = 5, N = 1024, \epsilon = 50h, \kappa_h = 2000$. The solute atom $X = {(0)}$ (i.e. single solute atom system), and the potential function $U(x;X)$ reads:
\[
U(x;X) = \rho_{\text{w}}\cdot \frac{1}{4\epsilon_{\text{LJ}}}\bigg[\Big(\frac{\sigma_0}{x_{\text{cut}}}\Big)^{12} - \Big(\frac{\sigma_0}{x_{\text{cut}}}\Big)^{6}\bigg] + \frac{Q^2}{8\pi\epsilon_0}\Big( \frac{1}{\epsilon_{\text{w}}} - \frac{1}{\epsilon_{\text{m}}}\Big)\frac{1}{x_{\text{cut}}^2}.
\]
Here $\rho_{\text{w}} = 0.0333 \mathring{\text{A}}^{-3}$ is the constant solvent (water) density, $\epsilon_{\text{LJ}} = 0.3 k_{\text{B}}T$ is the depth of the Lennard-Jones potential well associated with the solute atom, $\sigma_0 = 3.5 \mathring{A}$ is the finite distance at which the Lenard-Jones potential of the solute atom is zero, $x_{\text{cut}} = \max\{|x|,2.5\}$ is the cutoff distance of $x$ from solute atom, $Q = 1e$ is the partial charge of the solute atom, $\epsilon_0 = 1.4321\times 10^{-4} e^2/(k_{\text{B}}T\mathring{\text{A}})$ is the vacuum permittivity, $\epsilon_{\text{m}} = 1$ is the relative permittivity of the solute, and $\epsilon_{\text{w}} = 80$ is the relative permittivity of the solvent. See \cite{Zhao_2018CMS} for the model details.

Figure \ref{fig:pVISM} depicts the numerical equilibrium by taking $f(\phi) = 3\phi^2 - 2\phi^3$ and $f(\phi) = \phi$ in the pVISM system. One can see that for the model with $f(\phi) = 3\phi^2 - 2\phi^3$, the numerical equilibrium remains bounded between 0 and 1, the same as  the theoretical prediction. On the contrary, if $f(\phi) = \phi$, the numerical equilibrium becomes smaller than 0 inside the interface, and greater than 1 outside the interface. Of course, the violation of MPP can be mitigated by letting $\epsilon \rightarrow 0$ by the $\Gamma$-convergence theory \cite{LiZhao_SIAM2013}. However, in real applications, especially in the 3d simulations, $\epsilon$ has to remain relatively large to reduce the computational cost. Therefore, the choice of $f(\phi) = 3\phi^2 - 2\phi^3$ is advantageous of keeping the hyperbolic tangent profile of $\phi$, bounding $0\le \phi \le 1$ and localizing the forces only near the interfaces even for a relatively large $\epsilon$.

\begin{figure}[!htbp]
\centerline{
\includegraphics[width=150mm]{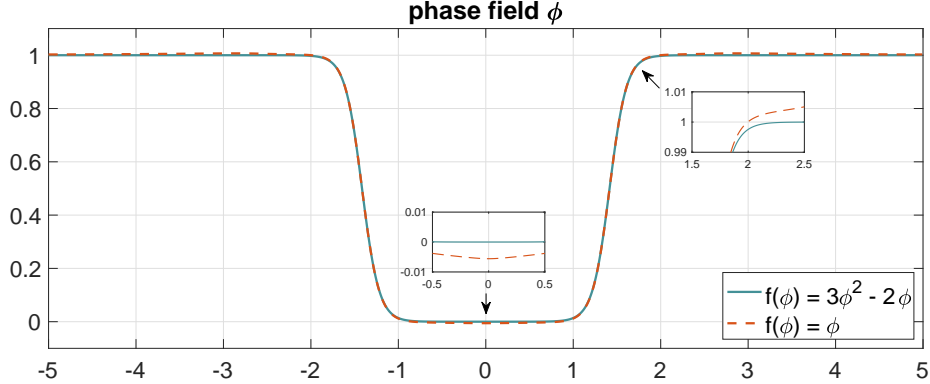}
 }
\caption{Numerical comparison between the model with $f(\phi) = 3\phi^2 - 2\phi^3$ and the one with $f(\phi) = \phi$ for the pVISM system. The model with $f(\phi) = 3\phi^2 - 2\phi^3$ holds the MPP, while the MPP is violated when taking $f(\phi) = \phi$.}
\label{fig:pVISM}
\end{figure}

\subsection{1D coarsening dynamics and MPP}

In this section, we verify the MPP and energy stability for the numerical scheme (\ref{eqn:pACOK_FullDiscreteII}) for the 1D case. We take a piecewise constant function as the initial, with the constant values generated randomly between 0 and 0.8. In the simulation, the parameter values are $T = 1000, \omega = 0.3, \gamma = 500, M = 2000, \tau = 10^{-3}, \kappa = 2000$. Figure \ref{fig:1dcoarsening_gamma500} shows the coarsening dynamics in which the system experiences phase separation from the random initial, then bumps appear from coarsening, evolve into same size, and finally are separated in equal distance.  The light blue curve (values labeled on the left $y$-axis) records the discrete $L^{\infty}$ norm for the solution $2\Phi^n-1$ (note that $|2\Phi^n-1|\le 1$ is equivalent to  $0\le \Phi^n \le 1$), which clearly implies the boundedness of $\Phi^n$ between 0 and 1. The red curve (values labeled on the right $y$-axis) represents the discrete energy $E_h^{\text{pOK}}[\Phi^n]$ in (\ref{eqn:discreteEnergy}) which is monotonically decreasing. Indicated by different colors, the four insets correspond to the four snapshots at $t = 0, 10, 500, 1000$ of the coarsening dynamics.

Now we fix all parameter values as they are in Figure \ref{fig:1dcoarsening_gamma500} but change $\gamma = 2000$, a larger value than it was. As $\gamma$ represents strength of the long-range repulsive interaction, we expect that a larger $\gamma$ generates more bumps. This is verified by Figure \ref{fig:1dcoarsening_gamma2000} in which the system still start from a randomly generated initial, but end up with six equally-sized equally-separated bumps. Meanwhile MPP and energy stability are still held as expected.

\begin{figure}[!htbp]
\centerline{
\includegraphics[width=150mm]{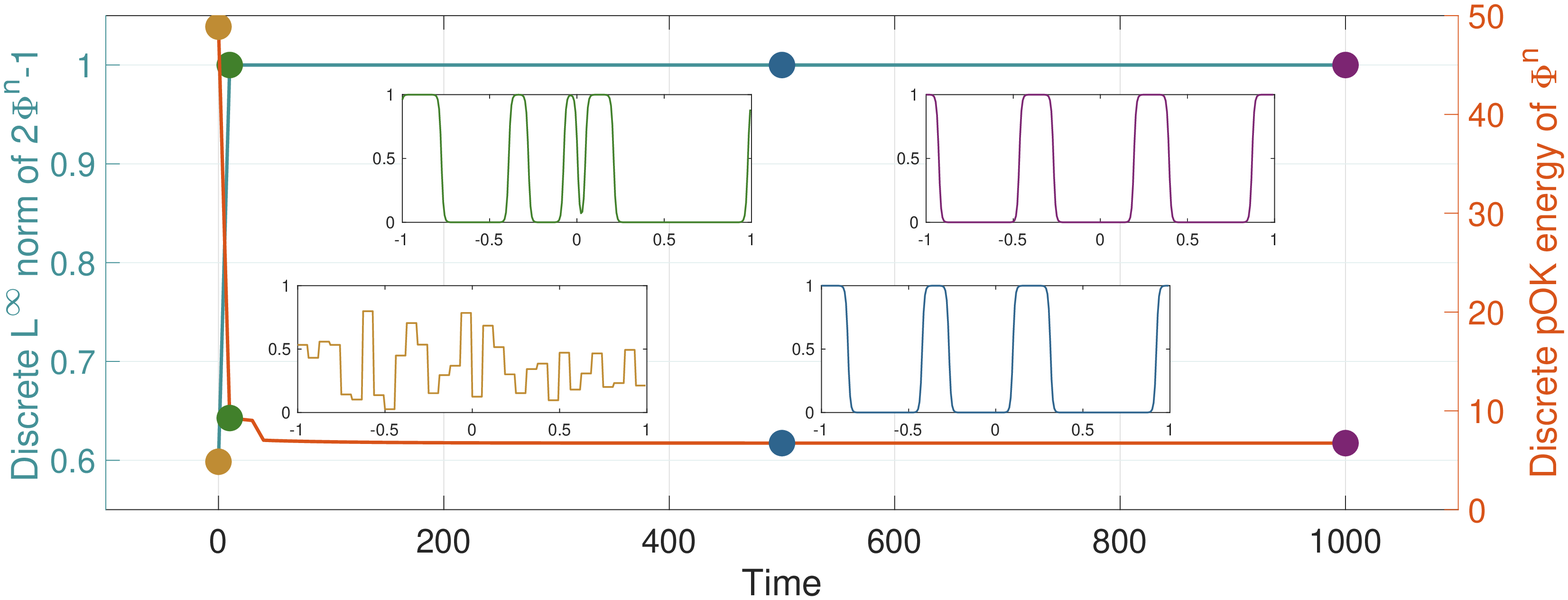}
 }
\caption{A 1D coarsening dynamics process with a small repulsive strength $\gamma$. In this simulation, the parameter values are $T = 1000, \omega = 0.3, \gamma = 500, M = 2000, \tau = 10^{-3}, \kappa = 2000$. The light blue curve records the discrete $L^{\infty}$ norm for the solution $\Phi^n$, which is clearly bounded between 0 and 1. The red curve represents the discrete energy $E_h^{\text{pOK}}[\Phi^n]$ in (\ref{eqn:discreteEnergy}) which is monotonically decreasing. The four insets are snapshots at different times.}
\label{fig:1dcoarsening_gamma500}
\end{figure}

\begin{figure}[!htbp]
\centerline{
\includegraphics[width=150mm]{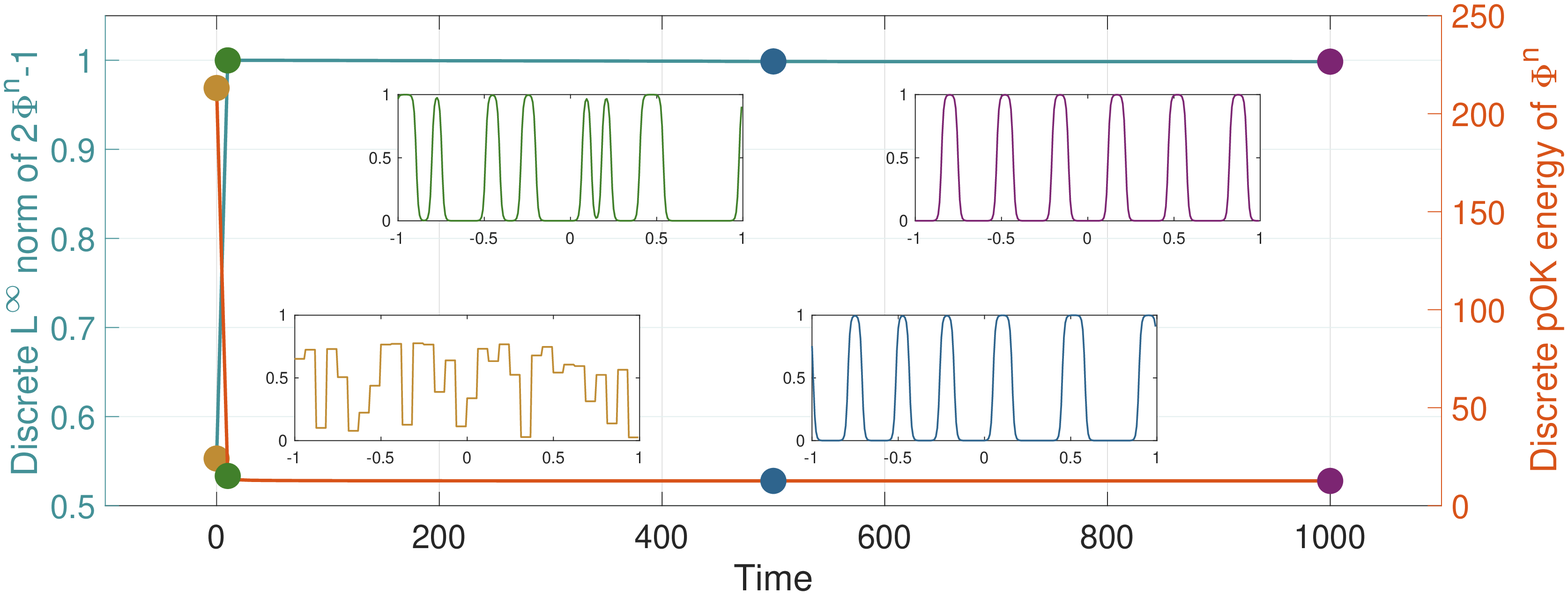}
 }
\caption{A 1D coarsening dynamics process with a large repulsive strength $\gamma$. In this simulation, the parameter values are $T = 1000, \omega = 0.3, \gamma = 2000, M = 2000, \tau = 10^{-3}, \kappa = 2000$.}
\label{fig:1dcoarsening_gamma2000}
\end{figure}

\subsection{2D coarsening dynamics and MPP}

In this section, we solve the equation (\ref{eqn:pACOK_FullDiscreteII}) in 2D and explore the corresponding discrete MPP and discrete energy stability.  We take a $256\times 256$ mesh grid and $T = 100, \omega = 0.15, \gamma = 2000, M = 10^{4}, \tau = 2\cdot 10^{-4}, \kappa = 2000$. Similar as in the 1D case, a 2D random initial is generated on a coarse grid. The coarsening dynamics is presented in Figure \ref{fig:2dcoarsening_gamma1000} in which the random initial is phase separated within a very short time period, resulting in a group of bubbles with different sizes, then the tiny bubbles disappear, other bubbles evolves into equal size, and eventually all the equally-sized bubbles become equally distanced, forming a hexagonal pattern in the 2D domain $\mathbb{T}^2$. Just like the 1D case, we see that the 2D coarsening dynamics also enjoy the MPP property and energy stability in the discrete sense as the theory predicts in the previous sections. The insets are snapshots taken at $t = 0, 1, 10, 100$, each of which has a colored title indicating the corresponding colored marker on the two curves.

When the value of $\gamma$ become larger, say $\gamma = 2000$, but other parameter values are fixed, the stronger long-range repulsive interaction between bubbles lead to more bubbles of equal size and equal distance. This result is depicted in Figure \ref{fig:2dcoarsening_gamma2000} in which the MPP and energy stability are still held.

\begin{figure}[!htbp]
\centerline{
\includegraphics[width=150mm]{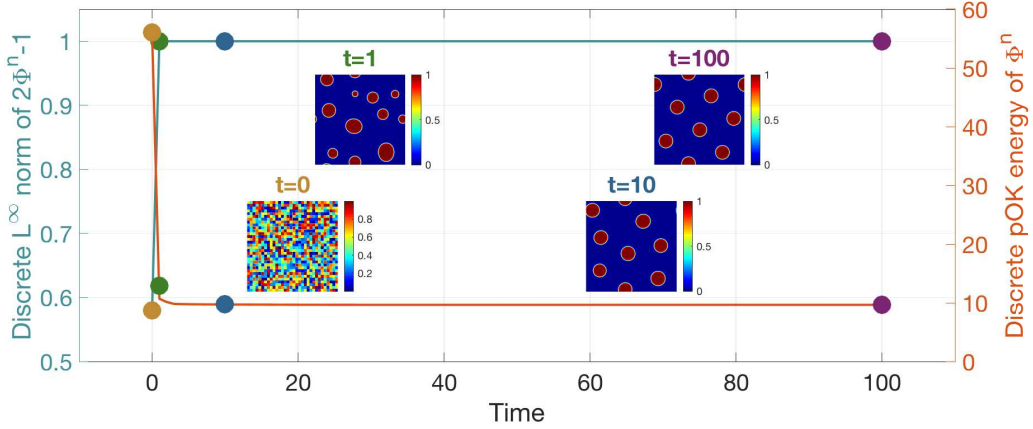}
 }
\caption{A 2D coarsening dynamics process with a small repulsive strength $\gamma$. In this simulation, the parameter values are $T = 100, \omega = 0.15, \gamma = 1000, M = 10^4, \tau = 2\cdot10^{-4}, \kappa = 2000$. The light blue curve is the discrete $L^{\infty}$ norm of $2\Phi^n-1$ which implies the bound of $\Phi^n$ between 0 and 1, while the red curve indicates the monotonic decay of the discrete energy $E_h^{\text{pOK}}$. The four insets are snapshots at different times. }
\label{fig:2dcoarsening_gamma1000}
\end{figure}

\begin{figure}[!htbp]
\centerline{
\includegraphics[width=150mm]{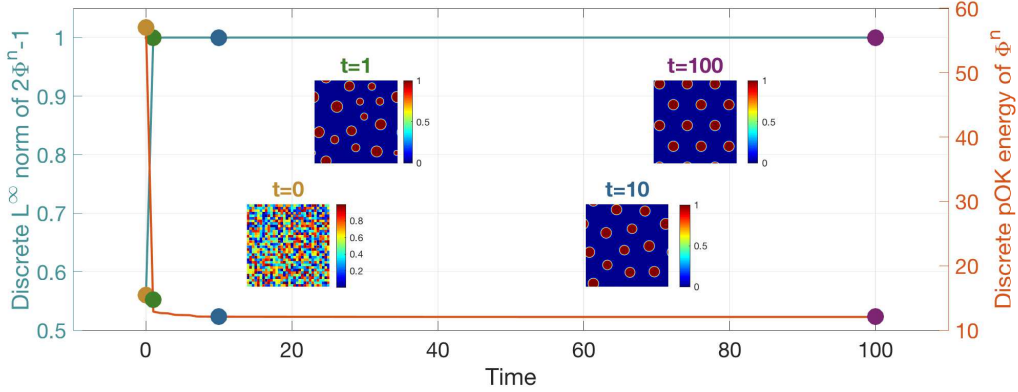}
 }
\caption{A 2D coarsening dynamics process with a large repulsive strength $\gamma = 2000$. Other parameters are the same as that for Figure \ref{fig:2dcoarsening_gamma1000}. Larger $\gamma$ lead to more bubbles forming hexagonal pattern.}
\label{fig:2dcoarsening_gamma2000}
\end{figure}

\section{Summary}

In this paper, we explore the MPP property for the pACOK equation and propose a first order stabilized linear semi-implicit scheme which inherits the MPP and the energy stability in the discrete level. The third order polynomial $f(\phi) = 3\phi^2 - 2\phi^3$ plays a key role in the proof of MPP for the system. We prove the MPP and energy stability in the semi-discrete and fully-discrete level  in which the nonlinear terms $W$ and $f$ need not to be extended to have bounded second order derivative.

In the numerical experiments, we test the rate of convergence for the proposed scheme. We also show that in some examples, a traditional choice of $f(\phi) = \phi$ could violate the MPP. When $\omega \ll 1$, the pACOK dynamics displays pattern of hexagonal bubble assemblies. When the repulsive long-range interaction becomes stronger, there will be more bubbles appearing in the hexagonal equilibria.

This work can be extended along several directions. Firstly, we can study for higher order MPP schemes for the pACOK equation, or generally binary systems with long-range interactions. Secondly, we can further consider the MPP scheme for ternary systems, or a more general system of $N+1$ constituents in which $N$ phase field functions $\{\phi_j\}_{j=1}^N$ are introduced to represent the densities of the $N$ constituents, and the $(N+1)$-th one is implicitly represented by $1- \sum_{j=1}^N \phi_j$.

In this paper, we mainly explore the numerical scheme for $L^2$ gradient flow dynamics based on operator splitting technique. Some other numerical methods, such as exponential time differencing based schemes, could be alternative choices for the MPP scheme, which will also be considered in the future.

\section{Appendix}\label{s:Appendix}

In the appendix, we briefly discuss the wellposedness of the pACOK dynamics (\ref{eqn:pACOK}) and the $L^{\infty}$ bound for the solution of (\ref{eqn:pACOK}).

\begin{definition}
Let $d = 2$ or 3. We call $\phi(t,x)$ a global weak solution to problem (\ref{eqn:pACOK}) if for any $T>0$, $\phi(t,x)$ satisfies
\[
\phi \in C([0,T]; L^p(\mathbb{T}^d)) \cap L^{\infty}(0,T; H^1(\mathbb{T}^d)) \cap L^{2}(0,T;H^2(\mathbb{T}^d)), \quad p \in [2,6)
\]
and the initial condition $\phi(0,x) = \phi_0(x)$. Further, for any $t\in(0,T]$, any test function $w\in L^2(\mathbb{T}^d)$, it holds
\begin{align*}
&\frac{\text{d}}{\text{d}t} \int_{\mathbb{T}^d} \phi(t,x) w(x) \ \text{d}x \\
= & \int_{\mathbb{T}^d} \bigg[\epsilon\Delta\phi - \dfrac{1}{\epsilon}W'(\phi) - \gamma(-\Delta)^{-1}(f(\phi)-\omega)f'(\phi) - M\int_{\mathbb{T}^d} (f(\phi)-\omega)\ \text{d}x\cdot f'(\phi) \bigg] w(x) \text{d}x
\end{align*}
in the distributional sense in $(0,T)$.
\end{definition}

With the definition of the weak solution for the problem (\ref{eqn:pACOK}), we are now ready to state the theorem for its wellposedness.

\begin{theorem}\label{theorem:wellposedness}
Let $d = 2$ or 3, and the initial data $\phi_0\in H^1(\mathbb{T}^d)$. Then there exists a unique global weak solution $\phi$ to the problem (\ref{eqn:pACOK}). Further, the free energy $E^{\emph{pOK}}$ in (\ref{functional:pOK}) decreases as time evolves.
\end{theorem}
The proof is a standard procedure by following De Giorgi's minimizing movement scheme \cite{Ambrosio_Rend1995,DeGiorgi_RMA1993}. We have a preprint discussing the wellposedness of a more complicated ternary system with long-range interaction, for which the proof of Theorem \ref{theorem:wellposedness} can be viewed as a straightforward application. Therefore we will omit the proof here and recommend the readers to refer to \cite{JooXuZhao_Preprint2020} for the details.

Our next result is regarding to the $L^{\infty}$ bound for the weak solution $\phi(t,x)$ of the problem (\ref{eqn:pACOK}), which can be achieved by De Giorgi's iteration \cite{Chen_Book2003,WuYinWang_Book2006}. Note that in this case, the result holds only for $d = 2$. To begin with, we need an algebraic lemma. Without causing any confusions, we point out that the notations $M, h, k, d, \alpha, \beta$ picked below are exclusively for Lemma \ref{lemma-algebra}, and might not mean the same as they are used elsewhere.

\begin{lemma}\label{lemma-algebra}
Let $\mu(t)$ be a nonnegative, non-increasing function on $[k_0, +\infty)$ that satisfies
\begin{equation}\label{growth-condition}
\mu(h)\leq\Big(\frac{M}{h-k}\Big)^{\alpha}\mu(k)^{\beta}, \quad\forall h>k\geq k_0,
\end{equation}
where $M>0$, $\alpha>0$, $\beta>1$ are all constants. Then we can find a constant $d>0$ such that
$$
  \mu(h)=0, \quad\forall h\geq k_0+d.
$$
\end{lemma}
\begin{proof}
To begin with, let
\[
k_s=k_0+d-\frac{d}{2^s}, \quad\forall s\in\mathbb{Z}^+,
\]
where $d$ is defined to be
\begin{align}\label{definition-d}
d=M2^{\frac{\beta}{\beta-1}}\mu(k_0)^{\frac{\beta-1}{\alpha}}.
\end{align}
Take $h = k_{s+1}$ and $k = k_s$ in (\ref{growth-condition}), it yields the recursive relation
\begin{equation}\label{recursive}
\mu(k_{s+1})\leq\frac{M^{\alpha}2^{(s+1)\alpha}}{d^{\alpha}}\mu(k_s)^{\beta}, \quad\forall s\in\mathbb{Z}^+.
\end{equation}

Next we claim that
\begin{equation}\label{claim}
\mu(k_s)\leq\frac{\mu(k_0)}{r^s}, \quad\forall s\in\mathbb{Z}^+,
\end{equation}
where $r>0$ is a constant defined as
\begin{equation}\label{definition-r}
r=2^{\frac{\alpha}{\beta-1}}>1.
\end{equation}
Once \eqref{claim} is verified, the proof of Lemma \ref{lemma-algebra} is done by simply passing $s\rightarrow \infty$
and using the assumption that $\mu$ is non-increasing.

We finally prove (\ref{claim}) by induction. Suppose (\ref{claim}) is valid for $s$, then we obtain from (\ref{definition-d}), (\ref{recursive}) and (\ref{definition-r}) that
\begin{align}
\mu(k_{s+1})\leq\frac{M^{\alpha}2^{(s+1)\alpha}}{d^{\alpha}}\frac{\mu(k_0)^\beta}{r^{\beta{s}}}
 = \frac{M^{\alpha}2^{(s+1)\alpha}}{M^{\alpha}r^{\beta}\mu(k_0)^{\beta-1}}\frac{\mu(k_0)^\beta}{r^{\beta{s}}}
 = \dfrac{\mu(k_0)}{r^{s+1}} \dfrac{r^{s+1}2^{(s+1)\alpha}}{r^{\beta(s+1)}}
 = \frac{\mu(k_0)}{r^{s+1}}
\end{align}
Hence \eqref{claim} is also valid if $s$ is replaced by $s+1$.
\end{proof}

Now we can present our result regarding to the $L^{\infty}$ bound for the weak solution to the problem (\ref{eqn:pACOK}) with initial data $\phi_0$ in 2D.

\begin{theorem}\label{theorem: Linfity}
For any $\phi_0\in H^1(\mathbb{T}^2)\cap L^\infty(\mathbb{T}^2)$ and $T>0$, the unique weak solution
\[
 \phi\in L^\infty(0, T; H^1(\mathbb{T}^2))\cap L^2(0, T; H^2(\mathbb{T}^2))
\]
to the problem (\ref{eqn:pACOK}) satisfies
\begin{equation}
\|\phi\|_{L^\infty([0,T]\times\TT)}\leq\|\phi_0\|_{L^\infty}+C^\ast,
\end{equation}
where $C^\ast>0$ is a constant that only depends on $\|\phi_0\|_{H^1}, \epsilon^{-1}, \omega, \gamma$ and $M$.
\end{theorem}
\begin{proof}
Let us denote
\[
  l=\|\phi_0\|_{L^\infty},
\]
the test function
\[
\xi(t,x)=(\phi(t,x)-k)^+\chi_{[t_1,t_2]},\quad\forall k>l\ \text{and}\ t_2>t_1
\]
and
\[
\tilde{F}(\phi)=-\frac{1}{\epsilon}W'(\phi)-\gamma(-\Delta)^{-1}\big(f(\phi)-\omega\big)f'(\phi)-M\Big(\int_{\TT}(f(\phi) - \omega)\,\ud{x}\Big)f'(\phi)
\]
Then it is immediate to check for any $p>2$, there exists $M_1>0$ that only depends on $p$, $\|\phi_0\|_{H^1}$, and coefficients of the equation, such that
\[
\|\tilde{F}(\phi(t))\|_{L^p}\leq M_1, \quad\forall t\in [0, T].
\]
Consider $\xi$ as a test function for the weak solution $\phi$, we obtain that
\begin{align}\label{integral-equality-1}
&\iint_{[0, T]\times\TT}\partial_t(\phi-k)^+(\phi-k)^+\chi_{[t_1,t_2]}\,\ud{x}\ud{t}+\iint_{[0, T]\times\TT}\big|\nabla(\phi-k)^+\big|^2\chi_{[t_1,t_2]}\,\ud{x}\ud{t}\nonumber\\
=&\iint_{[0, T]\times\TT}\tilde{F}(\phi)(\phi-k)^+\chi_{[t_1,t_2]}\,\ud{x}\ud{t}.
\end{align}
If we denote
\[
\I_k(t)=\int_{\TT}|(\phi(t,x)-k)^+|^2\,\ud{x},
\]
we get from \eqref{integral-equality-1} that
\begin{align}\label{integral-inequality-1}
\frac12\Big[\I_k(t_2)-\I_k(t_1)\Big]+\int_{t_1}^{t_2}\int_{\TT}\big|\nabla(\phi-k)^+\big|^2\,\ud{x}\ud{t}
\leq\int_{t_1}^{t_2}\int_{\TT}\big|\tilde{F}(\phi)\big|(\phi-k)^+\,\ud{x}\ud{t}.
\end{align}

Suppose $\I_k(t)$ attains its maximum value at $s\in [0, T]$ (assume $s>0$ without loss of generalization). Then
\[
\I_k(s)-\I_k(s-\eps)\geq 0,
\]
for any $0<\eps<s$, hence we derive from \eqref{integral-inequality-1} that
\begin{equation}\label{integral-inequality-2}
\int_{s-\eps}^{s}\int_{\TT}\big|\nabla(\phi-k)^+\big|^2\,\ud{x}\ud{t}
\leq\int_{s-\eps}^{s}\int_{\TT}\big|\tilde{F}(\phi)\big|(\phi-k)^+\,\ud{x}\ud{t}.
\end{equation}
Let us divide both sides of \eqref{integral-inequality-2} by $\eps$ and send $\eps\rightarrow 0^+$, it yields
\begin{equation}\label{key-inequality-1}
\int_{\TT}\big|\nabla(\phi(s,x)-k)^+\big|^2\,\ud{x}\leq \int_{\TT}\big|\tilde{F}(\phi(s,x))\big|(\phi(s,x)-k)^+\,\ud{x}
\end{equation}

Let us denote
$$
  \varphi(t,x)=(\phi(t,x)-k)^+, \quad\overline{\varphi}(t)=\frac{1}{|\TT|}\int_{\TT}\varphi(t,x)\,\ud{x}.
$$
By Sobolev embedding, Poincare's inequality and Young's inequality, we have
\begin{align}\label{Sobolev-embedding}
\Big(\int_{\TT}|\varphi(s, x)|^p\,\ud{x}\Big)^{\frac{2}{p}}
&\leq C\int_{\TT}|\varphi(s,x)|^2\,\ud{x}+C\int_{\TT}|\nabla\varphi(s,x)|^2\,\ud{x}\non\\
&\leq C\int_{\TT}|\varphi(s,x)-\overline{\varphi}(s)|^2\,dx+C\int_{\TT}|\overline{\varphi}(s)|^2\,\ud{x}+C\int_{\TT}|\nabla\varphi(s,x)|^2\,\ud{x}\non\\
&\leq C\int_{\TT}|\nabla\varphi(s,x)|^2\,\ud{x}+C\int_{\TT}|\varphi(s,x)|^2\,\ud{x},
\end{align}
where $C>0$ is a generic constant. If we further denote
$$
  F(t,x)=|\tilde{F}(\phi(t,x))|+|\varphi(t, x)|,
$$
then $\forall p>2$ we get
after combining \eqref{key-inequality-1} with \eqref{Sobolev-embedding} that
\begin{equation}\label{key-inequality-2}
\Big(\int_{\TT}|\varphi(s, x)|^p\,\ud{x}\Big)^{\frac{2}{p}}\leq C\int_{\TT}F(s, x)|\varphi(s,x)|\,\ud{x}.
\end{equation}
Note that
\begin{equation}\label{bound-F}
\|F(t)\|_{L^p}\leq M_2, \quad\forall t\in [0, T].
\end{equation}

Next, we denote
$$
  B_k(t)=\{x\in\TT: \phi(t, x)>k\}.
$$
Then it follows from \eqref{key-inequality-2} and  H\"{o}lder's inequality that
$$
  \Big(\int_{B_k(s)}|\varphi(s, x)|^p\,\ud{x}\Big)^{\frac{2}{p}}\leq C\int_{B_k(s)}F(s, x)|\varphi(s,x)|\,\ud{x}
  \leq C\Big(\int_{B_k(s)}|\varphi(s, x)|^p\,\ud{x}\Big)^{\frac{1}{p}}\Big(\int_{B_k(s)}|F(s, x)|^q\,\ud{x}\Big)^{\frac{1}{q}} ,
$$
where $q$ is the H\"older conjugate of $p$. It further implies
\begin{equation}\label{key-inequality-3}
\Big(\int_{B_k(s)}|\varphi(s, x)|^p\,\ud{x}\Big)^{\frac{1}{p}}\leq C\Big(\int_{B_k(s)}|F(s, x)|^q\,\ud{x}\Big)^{\frac{1}{q}},
\end{equation}
where $C>0$ only depends on $p$, $\|\phi_0\|_{H^1}$, and coefficients of equation (\ref{eqn:pACOK}). As a consequence, for any $1<m<p-1$ and its H\"older conjugate $m'$, it follows from \eqref{key-inequality-3} that
\begin{align}\label{key-inequality-4}
\Big(\int_{B_k(s)}|\varphi(s, x)|^p\,\ud{x}\Big)^{\frac{1}{p}}
\leq C\Big(\int_{B_k(s)} 1^m\,\ud{x}\Big)^{\frac{1}{mq}}\Big(\int_{\TT}|F(s, x)|^{m'q}\,\ud{x}\Big)^{\frac{1}{m'q}}
\leq C|B_k(s)|^{\frac{p-1}{mp}}
\end{align}

To conclude, on one hand, using H\"{o}lder inequality and \eqref{key-inequality-4} yield
\begin{align}\label{direction-1}
\I_k(t)\leq\I_k(s)\leq \Big(\int_{B_k(s)}|\varphi(s, x)|^p\,\ud{x}\Big)^{\frac{2}{p}}|B_k(s)|^{\frac{p-2}{p}}
\leq C|B_k(s)|^{\frac{2p-2}{mp}+\frac{p-2}{p}}, \quad\forall t\in [0, T].
\end{align}
On the other hand
\begin{align}\label{direction-2}
(h-k)^2|B_h(t)|\leq\int_{B_h(t)}\big|(\phi-k)^+\big|^2\,\ud{x}\leq\int_{B_k(t)}\big|(\phi-k)^+\big|^2\,\ud{x}\leq\I_k(t), \quad\forall h>k,\,t\in [0, T]
\end{align}
due to the fact that $\varphi\geq h-k$ on $B_h(t)$ and $B_h(t)\subset B_k(t)$. Therefore, if we denote
$$
  \mu(k)=\sup_{t\in[0, T]}|B_k(t)|,
$$
we get from \eqref{direction-1} and \eqref{direction-2} that
\begin{equation}\label{iterative-inequality}
\mu(h)\leq \Big(\frac{C}{h-k}\Big)^2\mu(k)^{\frac{2p-2}{mp}+\frac{p-2}{p}}.
\end{equation}
Note that
$$
  \frac{2p-2}{mp}+\frac{p-2}{p}>1
$$
by the choice of $m$, $p$, hence using Lemma \ref{lemma-algebra} we know that
$$
  \mu(l+C^\ast)=\sup_{t\in [0, T]}|B_{k+C^\ast}(t)|=0,
$$
which indicates
\begin{equation}
\phi(t, x)\leq l+C^\ast, \quad\forall (t, x)\in [0, T]\times\TT.
\end{equation}
Hence the proof is complete.

\end{proof}


\section{Acknowledgements}

X. Xu's work is supported by a grant from the Simons  Foundation through grant No. 635288;
Y. Zhao's work is supported by a grant from the Simons Foundation through Grant No. 357963.


\newpage


\bibliography{OhtaKawasaki}

\end{document}